\documentclass[12pt,leqno]{amsart}

\usepackage{amssymb, mathrsfs}

\usepackage[latin1]{inputenc}

\usepackage{tikz}
\usepackage{tikz-cd}
\usepackage{hyperref}
\usepackage[marginparwidth=4.5cm]{geometry}

\usepackage{color}

\newcommand{\mo}{\mathopen}
\newcommand{\mc}{\mathclose}

\newtheorem{theorem}{Theorem}[section]
\newtheorem{proposition}[theorem]{Proposition}
\newtheorem{lemma}[theorem]{Lemma}
\newtheorem{corollary}[theorem]{Corollary}

\theoremstyle{definition}

\newtheorem{definition}[theorem]{Definition}

\newtheorem{remark}[theorem]{Remark}

\newcommand{\C}{{\mathbb{C}}}
\newcommand{\N}{{\mathbb{N}}}
\newcommand{\R}{{\mathbb{R}}}

\newcommand{\Z}{{\mathbb{Z}}}
\newcommand{\HP}{{\mathbb{H}}}
\newcommand{\PP}{{\mathbb{P}}}
\newcommand{\TT}{{\mathbb{T}}}

\newcommand{\ak}{{\bf k}}

\renewcommand{\to}[1][]{\xrightarrow[]{#1}}
\newcommand{\from}[1][]{\xleftarrow[]{#1}}
\newcommand{\isoto}[1][]{\xrightarrow[#1]%
{{\raisebox{-.6ex}[0ex][-.6ex]{$\mspace{1mu}\sim\mspace{2mu}$}}}}
\newcommand{\isofrom}[1][]{\xleftarrow[#1]%
{{\raisebox{-.6ex}[0ex][-.6ex]{$\mspace{1mu}\sim\mspace{2mu}$}}}}

\newcommand{\RR}{\mathrm{R}}

\newcommand{\Hom}{\operatorname{Hom}}
\newcommand{\RHom}{\operatorname{RHom}}
\renewcommand{\hom}{\operatorname{\mathcal{H}om}}
\newcommand{\rhom}{\operatorname{R\mathcal{H}om}}

\newcommand{\rsect}{\operatorname{R\Gamma}}

\newcommand{\Der}{\operatorname{\mathsf{D}}}

\newcommand{\supp}{\operatorname{supp}}
\newcommand{\msupp}{\operatorname{SS}}
\newcommand{\dmsupp}{\operatorname{\dot{SS}}}

\newcommand{\dT}{\dot{T}}

\newcommand{\id}{\mathrm{id}}

\newcommand{\ol}{\overline}

\newcommand{\sui}{\medskip\noindent}
\newcommand{\suiv}[1]{\smallskip\noindent{\rm #1}}

\newcommand{\gp}{\mathsf{G}}
\newcommand{\sgp}{\mathsf{H}}
\newcommand{\calF}{\mathcal{F}}
\newcommand{\calT}{\mathcal{T}}

\newcommand{\lc}{\mathrm{lc}}
\newcommand{\Dlc}{\operatorname{\mathsf{D}}_{\mathrm{lc}}}

\title{Viterbo's spectral bound conjecture for homogeneous spaces}

\author{St\'ephane Guillermou}
\thanks{ SG: UMR 5582 du CNRS Institut Fourier, Université Grenoble Alpes,  CS 40700
 38058 Grenoble cedex 9 - France}

\author{Nicolas Vichery}
\thanks{NV:  UMR 5208 du CNRS Institut Camille Jordan, Universit\'e Claude Bernard Lyon 1,
 43 boulevard du 11 novembre 1918, 69622 Villeurbanne Cedex - France}

\thanks{Both authors acknowledge support from ANR MICROLOCAL (ANR-15-CE40-0007) and ANR COSY
  (ANR-21-CE40-0002). Part of this article has been written during a stay of Nicolas Vichery at
  Institut Fourier.}

\begin{document}
\maketitle

\begin{abstract}
  We prove a conjecture of Viterbo about the spectral distance on the space of compact exact
  Lagrangian submanifolds of a cotangent bundle $T^*M$ in the case where $M$ is a compact
  homogeneous space: if such a Lagrangian submanifold is contained in the unit ball bundle of
  $T^*M$, its spectral distance to the zero section is uniformly bounded.  This also holds for some
  immersed Lagrangian submanifolds if we take into account the length of the maximal Reeb chord.  
\end{abstract}

\section{Introduction}

In 2007, Viterbo conjectured  the following result about the spectral distance on exact Lagrangian submanifolds which he defined in 1982 \cite{Vit92}  (this is mentioned in \cite{Vit08}). Every Lagrangian submanifold
of  $T^*\TT^n$,
Hamiltonian isotopic to the zero section $0_{\TT^n}$, and included into the unit codisc bundle of $T^*\TT^n$ is at a  distance uniformly bounded from $0_{\TT^n}$.
This could be understood as a tentative to find non trivial compact sets for the spectral distance or at least bounded sets, which is a completely open question.
The conjecture has been since generalized for every cotangent bundle of compact manifold and not restricted to Lagrangian isotopic to the zero section. Note that a priori the spectral norm depends on the coefficient ring. We will restrict ourselves to fields.

The main set of applications are in relation with the existence and properties of quasi-morphisms on the Hamiltonian group of cotangent bundles and with Hamiltonian dynamic. It has been anticipated in symplectic homogenization \cite{Vit08}, \cite{MVZ}.

Biran and Cornea \cite{BR} have proved some bound on this distance but depending on the boundary depth of the Lagrangian with any fiber. Conversely, these authors have shown that this boundary depth is bounded in case the Viterbo conjecture holds.

Shelukhin in 2018  \cite{S18} and 2019 \cite{S19} solves the conjecture in the following cases, with two different approaches: 

\begin{enumerate}
    \item The case of $M= \R\PP^n$, $\C\PP^n$, $\HP\PP^n$, $\mathbb{S}^n$ for the field $\mathbb{F}_2$ which has been obtained by SFT techniques applied to closed symplectic manifold.
    \item The case of $M$ ``string point invertible'' which depends also of the chosen field and
      contains compact Lie groups according to previous computations of Menichi. This notion
      introduced by Shelukhin is related to the possibility to construct the fundamental class of
      the underlying manifold via cohomological classes of the free loop space.
\end{enumerate}

During the preparation of this paper we learned that Viterbo~\cite{V22} had another proof of the
conjecture for manifolds satisfying a cohomological condition, whose typical example is a
homogeneous space.

Dimitroglou Rizell proved in  \cite{D} that the conjecture is false in the case of immersed Lagrangian by counterexample in $T^*\mathbb{S}^1$.

\medskip

Here, we prove the conjecture in the case of compact homogeneous spaces.  First, in the case of
compact Lie groups we give a statement valid for any coefficient field in Theorem~\ref{thm:groupe}.
In the case of a general compact homogeneous space we prove Corollary~\ref{cor:esphomogsimplconnexe}
with some restriction on the characteristic of the field and with a weaker bound for the spectral
distance.  Moreover we extend the initial conjecture in two directions: (1) the proof holds for any
finite ranked local system over the Lagrangian; (2) we take into account the counterexample of
Dimitroglou Rizell by giving a bound depending on the length of the maximal Reeb chord in the case
of an immersed Lagrangian for which there exists a ``quantization'' (a sheaf whose microsupport is
our Lagrangian -- such a sheaf exists for a Legendrian deformation of the zero section, see
Corollary~\ref{cor:defLegendrienne}).  A byproduct of the proof is a bound for the boundary depth of
the Lagrangian with the zero section.

\subsection*{Idea of proof}
Let $M$ be a compact manifold and $\Lambda \subset T^*M$ an exact compact Lagrangian submanifold.
We use the microlocal theory of sheaves, which deals with conic subsets of the cotangent bundle, and
we lift $\Lambda$ to a conic Lagrangian submanifold $\Lambda'\subset T^*(M\times \R) \setminus
0_{M\times\R}$. According to \cite{G19}, there exists a sheaf $F$ with $\msupp(F) = \Lambda'$ which
is isomorphic to zero near $M\times\{-\infty\}$ and to $\ak_{M\times\R}$ near $M\times\{+\infty\}$.

Tamarkin translates the filtration of the Floer complex into a sheaf morphism $\tau_c(F) \colon F
\to T_{c*}F$ for $c\geq 0$, where $T_c$ is the translation by $c$ along $\R$. He also introduces
a twisted $\hom$-sheaf, $\hom^*$, with the property that $\RR a_* \hom^*(F,G)$ encodes the
filtration on the Floer cohomology, with $a$ the projection to $\R$.

We first use a computation of $\hom^*$ on sheaves over $\R$. We deduce that $F' = \hom^*(\RR a_*F,
\RR a_*F) \otimes \ak_{[0,+\infty[}$ contains the information on the Viterbo norm,
$\gamma(\Lambda,0_M)$, and the boundary depth (see Lemma~\ref{lem:c_et_v}).  More precisely
$\tau_c(F')$ is non zero as long as $c<\gamma(\Lambda,0_M)$.

We use the fact that $M=\gp$ is a group by seeing $\hom^*(\RR a_*F_1, \RR a_*F_2)$, for any two
sheaves $F_1,F_2$, as the ``average'' of $\RR a_*\hom^*(F_1, \mu_g^! F_2)$ for $g\in \gp$. More
precisely we define
$$
\hom^{*,\gp}(F_1,F_2) = \RR p_{1*}\rhom(p_2^{-1}F, \mu^!F') ,
$$
where $p_1,p_2, \mu \colon (\gp\times\R)^2 \to \gp\times\R $ are the projections and the group
operation, and we have
$\RR a_* \hom^{*,\gp}(F_1,F_2) \simeq \hom^*(\RR a_*F_1, \RR a_*F_2)$.  Now, when $F_1 = F_2
=F'$, we can see that the morphism $\tau_c(F')$ we are interested in is the image by $a_*$ of the
morphism $\tau_c(F'')$ with $F'' = \hom^{*,\gp}(F_1,F_2)\otimes \ak_{M\times[0,+\infty[}$.  We can
see that $F''|_{\{e\} \times \R}$ vanishes outside $[0,l_{max}]$ where $l_{max}$ is the length of
the maximal Reeb chord of $\Lambda$ (see Lemma~\ref{lem:annulationF2fibree}).

These morphisms $\tau_c$ can be interpreted as sections of sheaves and the vanishing of $\tau_c$ can
be translated as the vanishing of some sections over subsets of the form $M\times \mo]-\infty,-c[$.

Now the microlocal theory of sheaves gives the following general propagation argument. If $G$ is a
sheaf on $M\times\R$ with $\msupp(G)$ contained in a cone $\{\tau \geq ||\xi||\}$, then the
restriction map from the cone $\{(x,t)$; $t<-d(x_0,x) < r\}$, where $d$ is the distance and $r$ the
injectivity radius, to the half-line $\{x_0\} \times \mo]-\infty,0[$ is an isomorphism (actually we
need a slightly more technical argument -- see Proposition~\ref{prop:annulation_section} and
Corollary~\ref{cor:annulationtauc}).

We use this argument to prove the vanishing of $\tau_{2r+l_{max}}(F'')$ in a ball of radius $r$
around the identity of $\gp$ and, using the group action, we obtain the same vanishing over any ball
of radius $r$ of $\gp$. By a general d\'evissage argument on a triangulation we extend the vanishing
of $\tau_c(F'')$ on balls to a global vanishing of $\tau_{c\dim \gp}(F'')$ (we have to be careful
that the step from the $k$-skeleton to the $(k+1)$-skeleton requires to change $\tau_{kc}$ into
$\tau_{(k+1)c}$). This proves the Lie group case.

We deduce the case of a homogeneous space as follows.  Let $q \colon \gp\times\R \to
\gp/\sgp\times\R = M\times\R$ be the quotient map.  For $\Lambda \subset T^*M$ we have a natural
inverse image $\Lambda' \subset T^*\gp$. It is easy to see that the spectral norm of $\Lambda'$ is
smaller than the one of $\Lambda$, but they are a priori not equal.  However in the other direction
we have equality: if $F$ is a sheaf on $\gp\times\R$ with microsupport $\Lambda'$, then the spectral
norms of $F$ and $\RR q_*F$ are equal.  Now we can see that any sheaf $G$ on $M\times\R$ with
microsupport $\Lambda$ is a direct summand of some iterated cone between such direct image sheaves
$\RR q_* F$ (it is even enough to take $m$ cones, where $m$ is the dimension of $M$).  We note that
each time we take a cone we need to change some $\tau_{kc}$ into $\tau_{(k+1)c}$ and thus get a
weaker bound.

\subsection*{Acknowledgements}

The authors thank Sylvain Courte and Claude Viterbo for useful conversations during the preparation
of this paper.

\newpage

\subsection*{Notations}
We follow mostly the notations of~\cite{KS90}.

Let $M$ be a manifold.  Our coefficient ring is a field $\ak$.  We denote by
$\Der(\ak_M)$ the derived category of sheaves of $\ak$-vector spaces on $M$.  We
usually work on a product $M\times\R$ and denote by $(t;\tau)$ the coordinates on
$T^*\R$.  The microsupport of a sheaf $F$ on $M$ is written $\msupp(F)$; it is a
closed conic subset of $T^*M$. We set $\dmsupp(F) = \msupp(F) \setminus 0_M$.  We let
$\Der_{\tau\geq 0}(\ak_{M\times\R})$ be the full subcategory of
$\Der(\ak_{M\times\R})$ formed by the $F$ such that
$\msupp(F) \subset T^*M \times \{\tau\geq 0\}$.

We recall the bounds given in Propositions~5.4.4-13-14 of~\cite{KS90}.  Let
$f\colon M\to M'$ and $F,G\in \Der(\ak_M)$, $F' \in \Der(\ak_{M'})$ be given.  Let
$T^*M \from[f_d] M\times_{M'}T^*M' \to[f_\pi] T^*M'$ be the natural maps.  Assuming
respectively (i) $f$ is proper on $\supp(F)$, (ii)
$\dmsupp(F') \cap f_\pi(f_d^{-1}(0_M)) = \emptyset$, (iii)
$\dmsupp(F) \cap \dmsupp(G)^a = \emptyset$ and (iv)
$\dmsupp(F) \cap \dmsupp(G) = \emptyset$, we have
\begin{equation}
  \label{eq:borne-msupp}
  \begin{gathered}
    \msupp(\RR f_*F) \subset f_\pi(f_d^{-1}(\msupp(F))), \hspace{2cm}
    \msupp(f^{-1}F') \subset f_d(f_\pi^{-1}(\msupp(F'))), \\
    \msupp(F\otimes G) \subset \msupp(F) + \msupp(G), \hspace{1cm}
    \msupp(\rhom(F, G)) \subset \msupp(F)^a + \msupp(G).
\end{gathered}
\end{equation}

\section{Reminder on Tamarkin's framework}
\subsection{On Tamarkin's morphism}
In this section we often use the maps
$$
s, q_1, q_2 \colon M\times \R^2 \to M\times \R,
\qquad T_c \colon M\times\R \to M\times\R,
$$
where $s, q_1, q_2$ send $(x,t,t')$ respectively to $(x,t+t')$, $(x,t)$, $(x,t')$ and
$T_c(x,t) = (x,t+c)$.

Following~\cite{KS90} we define the functor
$P \colon \Der(\ak_{M\times\R}) \to \Der(\ak_{M\times\R})$ by
$P(F) = \RR s_*(F \boxtimes \ak_{[0,+\infty[})$.  We remark that
$\RR s_*(F \boxtimes \ak_{\{0\}})$ is naturally isomorphic to $F$. Hence the morphism
$\ak_{[0,+\infty[} \to \ak_{\{0\}}$ gives a natural morphism of functors $P \to \id$.
By~\cite{KS90} we know that $F \in \Der_{\tau\geq 0}(\ak_{M\times\R})$ if and only if
this morphism $P(F) \to F$ is an isomorphism.

More generally, we have a natural isomorphism
$\RR s_*(F \boxtimes \ak_{\{c\}}) \simeq T_{c*}F$.  If $c\geq 0$ we also have a
morphism $\ak_{[0,+\infty[} \to \ak_{\{c\}}$ and we deduce a natural morphism of
functors $\tau_c \colon P \to T_{c*}$.
\begin{definition}\label{def:tauc}
For $F \in \Der_{\tau\geq 0}(\ak_{M\times\R})$ we thus obtain what we call the
Tamarkin's morphism
\begin{equation}
  \label{eq:deftau}
  \tau_c(F) \colon F \to T_{c*}F, \qquad \text{for $c\geq0$.}
\end{equation}
We say that a $F$ is {\em torsion} if $\tau_c(F) = 0$ for some $c\geq0$.
\end{definition}

The following lemma follows from results of~\cite{GS14} (it is a quantitative version
of Lem.~6.3 with the same proof).

\begin{lemma}\label{lem:taucducone}
  Let $F,F',G \in \Der_{\tau\geq 0}(\ak_{M\times\R})$. We assume that we have a
  distinguished triangle $F\to G \to F' \to[+1]$ and that $\tau_a(F) =0$,
  $\tau_b(F') =0$, for some $a,b\geq 0$.  Then $\tau_{a+b}(G) =0$.
\end{lemma}
\begin{proof}
  By the functorial properties of the morphism $\tau_c$ (see~\cite[\S6]{GS14}) we
  have the following commutative diagram of distinguished triangles
\begin{equation*}
  \begin{tikzcd}[column sep = 17mm, row sep = 8mm]
    F \ar[r] \ar[d, "\tau_b(F)"']
    & G \ar[r]  \ar[d, "\tau_b(G)"']\ar[dl, dotted, "u"']
    & F' \ar[r,"+1"]  \ar[d, "\tau_b(F')"'] & {} \\
    T_{b*}F \ar[r, "\alpha"] \ar[d, "T_{b*}(\tau_a(F))"']
    & T_{b*}G \ar[r, "\beta"]  \ar[d, "T_{b*}(\tau_a(G))"']
& T_{b*}F' \ar[r,"+1"]  \ar[d, "T_{b*}(\tau_a(F'))"'] & {}\\
T_{(a+b)*}F \ar[r, "\alpha'"] & T_{(a+b)*}G \ar[r] & T_{(a+b)*}F' \ar[r,"+1"] 
    & {}\rlap{,}
\end{tikzcd}
\end{equation*}
where the vertical morphisms are the morphisms $\tau_c$ of the corresponding sheaves
or their images by $T_{b*}$.  We also have
$\tau_{a+b}(H) = T_{b*}(\tau_a(H)) \circ \tau_b(H)$, for any
$H \in \Der_{\tau\geq 0}(\ak_{M\times\R})$.

Since $\tau_b(F') =0$, we have $\beta \circ \tau_{b}(G) =0$ and we can factorize
$\tau_{b}(G) = \alpha \circ u$ for some $u\colon G \to T_{b*}F$ (this follows from
the general fact that $\Hom(G,-)$ turns distinguished triangles into long exact
sequences). Hence
$$
\tau_{a+b}(G) = T_{b*}(\tau_a(G)) \circ \tau_b(G)
= \alpha' \circ T_{b*}(\tau_a(F)) \circ  u
$$
and this vanishes since $\tau_a(F)=0$.
\end{proof}

We will later assume that $M$ is endowed with a metric and we will give conditions so
that the vanishing of $\tau_c(F)|_{\{x_0\}\times \R}$, for some $x_0 \in M$, implies
the vanishing of $\tau_{c+d}(F)|_{B\times \R}$, for a ball of radius $d$ around
$x_0$.  For this it is convenient to consider $\tau_c$ as a section of some sheaf and
use microsupport estimates to extend this section.  Such a sheaf was defined by
Tamarkin as follows.  For $F, G \in \Der(\ak_{M\times\R})$ we set
\begin{equation}
  \label{eq:defhometoile}
\hom^*(F,G) = \RR q_{1*} \rhom( q_{2}^{-1} F, s^! G) .
\end{equation}
We then have

\begin{lemma}\label{lem:section_hom_etoile}
  Let $U \subset M$ be an open subset and $F, G \in \Der(\ak_{M\times\R})$. \\ Then
  $\hom^*(F,G)|_{U\times\R} \simeq \hom^*(F|_{U\times\R},G|_{U\times\R})$ and, for
  any $c\in\R$,
  $$
  \rsect_{U\times \{-c\}}(U\times \R; \hom^*(F,G))
  \simeq \RHom(F|_{U\times\R}, T_{c*}(G|_{U\times\R})) .
  $$
\end{lemma}
\begin{proof}
  The fist assertion is obvious. For the second isomorphism, the adjunction between
  $*$ and $\hom^*$ gives
\begin{align*}
  \rsect_{U\times \{-c\}}(U\times \R; \hom^*(F,G))
  &\simeq \RHom( \ak_{U\times\{-c\}}, \hom^*(F,G))  \\
  &\simeq \RHom( \ak_{U\times\{-c\}} * F,G)  \\
  &\simeq \RHom( T_{-c*} F, G) .
\end{align*}
\end{proof}

\begin{lemma}\label{lem:annulationR}
  Let $F\in \Der_{\tau\geq0}(\ak_{M\times\R})$ such that the map $\supp(F) \to M$ is
  proper.  Then $\rsect(M\times\R;F) \simeq 0$.
\end{lemma}
\begin{proof}
  Let $p\colon M\times\R \to M$ be the projection.  It is enough to see that, for any
  $x\in M$, the stalk $(\RR p_*F)_x \simeq \rsect(\{x\}\times\R; F|_{\{x\}\times\R})$
  vanishes.  We choose $a<b$ with $\supp(F|_{\{x\}\times\R}) \subset [a,b]$. We have
$$
    \rsect(\R; F|_{\{x\}\times\R}) \isoto \rsect(\mo]-\infty,b+1[; F|_{\{x\}\times\R}) 
\isoto  \rsect(\mo]-\infty,a-1[; F|_{\{x\}\times\R})  \simeq 0 ,
$$
where the second isomorphism follows from the Morse result~\cite[Cor.~5.4.19]{KS90}
(applied with the function $\phi(t) = t$).
\end{proof}

\begin{lemma}\label{lem:propag1}
  Let $F \in \Der_{\tau\geq0}(\ak_{M\times\R})$.  Then, for any $a<b$, we have
  $$\rsect(M\times\mo]a,b[; F)[-1] \isoto \rsect_{M\times \{b\}}(M\times \R; F).$$
\end{lemma}
\begin{proof}
  Applying $\RHom(-,F)$ to the distinguished triangle
  $\ak_{M\times ]a,b[} \to \ak_{M\times ]a,b]} \to \ak_{M\times \{b\}} \to
  \ak_{M\times ]a,b[}[1]$ shows that the cone of the morphism of the lemma is
  $\rsect(M\times \R; F')$, where $F' = \rhom(\ak_{M\times ]a,b]} , F)$. By
  Lemma~\ref{lem:annulationR} this cone vanishes.
\end{proof}

We will have to consider sheaves $F$ on $\R$ for which it is easier to compute the
``costalks'' $\rsect_{\{t\}}(F)$ than directly the stalks $F_t$ and the next result
will be useful.
\begin{lemma}\label{lem:cogermeszero}
  Let $F\in \Der_{\tau\geq0}(\ak_\R)$. We assume $\rsect_{\{t\}}(F) \simeq 0$ for all
  $t>0$. Then $F|_{[0,\infty[} \simeq 0$.
\end{lemma}
\begin{proof}
  We have $H^iF_t \simeq \varinjlim_{\varepsilon>0} H^i(]t-\varepsilon,t+\varepsilon[; F)$
  (see~\cite[Rem.~2.6.9]{KS90}).  By Lemma~\ref{lem:propag1} we have
  $$
  \rsect(]t-\varepsilon,t+\varepsilon[; F) \simeq \rsect_{\{t+\varepsilon\}}(\R;F)
  \simeq \rsect(\R; \rsect_{\{t+\varepsilon\}}F))
  $$
  and this vanishes for all $t\geq 0$ and $\varepsilon>0$ by the hypotheses of the
  lemma. Hence $H^iF_t \simeq 0$ for all $t\geq0$ and $i\in\Z$ and the result
  follows.
\end{proof}

\subsection{Viterbo spectral invariants and boundary depth}
\label{sec:Vsi}

In this section we recall the definition of Viterbo spectral invariants from the
point of view of sheaf theory.  Let $\ol \Lambda \subset T^*M$ be a compact exact
Lagrangian submanifold and let $\Lambda \subset T^*_{\tau>0}(M\times\R)$ be a
$\R_{>0}$-conic lift of $\ol\Lambda$.  We know by~\cite{G19} or~\cite{V19} that there
exists $F \in \Der(\ak_{M\times\R})$ such that $\dmsupp(F) = \Lambda$,
$F|_{M \times \{t\}} \simeq 0$. Hence it is natural to consider the following
conditions, for a general $F \in \Der(\ak_{M\times\R})$:
\begin{equation}
  \label{eq:cond_faiscMR}
\left\{  \begin{aligned}
& \dmsupp(F) \subset T^*(M\times \mo]-A,A[) \cap \{\tau\geq 0\} ,\text{ for some $A>0$,} \\
& F|_{M \times \{t\}} \simeq 0, \qquad \text{for $t < -A$}, \\
\end{aligned} \right.
\end{equation}
The sheaf $F|_{M\times \mo]A,+\infty[}$ is then locally constant; hence
\begin{equation}
  \label{eq:defF+}
F_+ = F|_{M \times \{t\}}, \qquad \text{ for $t>A$,}   
\end{equation}
is well-defined.

We will in fact consider more general $\R_{>0}$-conic Lagrangian submanifold of
$\Lambda \subset T^*_{\tau>0}(M\times\R)$, namely those coming from an {\em immersed}
compact exact Lagrangian submanifold $\ol \Lambda$ of $T^*M$.  We denote by
\begin{equation}\label{eq:defDerLambdaplus}
\Der_{\Lambda,+}(\ak_{M\times\R})
\end{equation}
the subcategory of $\Der(\ak_{M\times\R})$ formed by the $F$ such that
$\dmsupp(F) = \Lambda$ and $F|_{M \times \{t\}} \simeq 0$ for $t \ll 0$ (in
particular $F$ satisfies~\eqref{eq:cond_faiscMR} because the compactness of
$\ol \Lambda$ implies $\Lambda \subset T^*(M\times \mo]-A,A[)$ for some $A>0$).

\begin{remark}\label{lem:equiv-cat}
  (i) In the case where $\ol\Lambda$ is embedded we have a uniqueness statement
  in~\cite{G19} or~\cite{V19}. Let $\Dlc(\ak_M)$ be the subcategory of $\Der(\ak_M)$
  formed by the $G$ such that $H^iG$ is a locally constant sheaf, for any $i\in\Z$.
  Then $\Der_{\Lambda,+}(\ak_{M\times\R}) \to \Dlc(\ak_M)$, $F\mapsto F_+$, is an
  equivalence of categories. In particular there exists a unique
  $F \in \Der_{\Lambda,+}(\ak_{M\times\R})$ such that $F_+ \simeq \ak_M$.

  \smallskip\noindent (ii) By the results of~\cite{GKS12} we have an immediate
  generalization of this equivalence if we replace $\Lambda$ by a Legendrian
  deformation $\Lambda_1$ of $\Lambda$ (here we identify $\R_{>0}$-conic Lagrangian
  submanifolds of $T^*_{\tau>0}(M\times\R)$ with Legendrian submanifolds of
  $J^1(M)$).  Indeed, for such a deformation, \cite{GKS12} gives an equivalence
  $\Der_{\Lambda}(\ak_{M\times\R}) \simeq \Der_{\Lambda_1}(\ak_{M\times\R})$ and we
  deduce easily that $\Der_{\Lambda_1,+}(\ak_{M\times\R}) \to \Der(\ak_M)$,
  $F\mapsto F_+$, is again an equivalence.
\end{remark}

\begin{lemma}
  For $F \in \Der(\ak_{M\times\R})$ satisfying~\eqref{eq:cond_faiscMR} we have
  $$
  \rsect(M\times\R; F) \isoto \rsect(M\times \mo]t;+\infty[; F) \isoto \rsect(M;F_+)
  $$
for all $t$ and
\begin{alignat*}{2}
 \rsect(M\times \mo]-\infty,t[; F) &\isofrom  \rsect(M\times\R; F) &\qquad &\text{for $t>A$}, \\
  \rsect(M\times \mo]-\infty,t[; F)&\simeq 0 && \text{for $t<-A$}.
\end{alignat*}
\end{lemma}
\begin{proof}
  The cone of the first morphism is $\rsect(M\times\R; \rsect_{\mo]-\infty,t]}(F))$,
  which vanishes by Lemma~\ref{lem:annulationR} (the support is proper by the
  hypothesis~\eqref{eq:cond_faiscMR}).  The second isomorphism is clear when $t>A$
  because $F|_{M\times \mo]A,+\infty[}$ is locally constant.  The third isomorphism
  also reduces to Lemma~\ref{lem:annulationR}, using that
  $F|_{M\times \mo]A,+\infty[}$ is locally constant. The last isomorphism is clear
  by~\eqref{eq:cond_faiscMR}.
\end{proof}

\begin{definition}
  For $F \in \Der(\ak_{M\times\R})$ satisfying~\eqref{eq:cond_faiscMR} and
  $\alpha \in H^\star(M;F_+)$ we set $c(F,\alpha) = \sup\{t$;
  $\alpha|_{M\times \mo]-\infty,t[} = 0\}$.  We set $c_-(F) = \min\{ c(F,\alpha)$;
  $\alpha \in H^\star(M;F_+)\}$ and $c_+(F) = \max\{ c(F,\alpha)$;
  $\alpha \in H^\star(M;F_+)\}$.
\end{definition}

\begin{lemma}\label{lem:cas_F_canonique}
  We assume that $\Lambda \subset T^*_{\tau>0}(M\times\R)$ is as in
  Remark~\ref{lem:equiv-cat},(i) or (ii), and we let
  $F \in \Der_{\Lambda,+}(\ak_{M\times\R})$ be the unique sheaf with
  $F_+ \simeq \ak_M$. Then, for $\alpha \in H^i(M;\ak_M)$, we have
  $c(F,\alpha) = c(\Lambda,\alpha)$ (the usual Viterbo invariants of $\Lambda$).
\end{lemma}
In particular we have $c_-(F) = c(\Lambda,1)$ with $1 \in H^0(M;\ak_M) \simeq \ak$
and, if $M$ is oriented or $\ak$ is of characteristic $2$, we have
$c_+(F) = c(\Lambda,\delta_M)$, where $\delta_M \in H^n(M;\ak_M)$ is the fundamental
class, $n=\dim M$.
  
\begin{proof}
  If we assume that $\Lambda$ has a generating function quadratic at infinity, say
  $f\colon M \times \R^N \to \R$, we have an easy construction of the sheaf as
  $F = \RR q_*(\ak_{\{t\geq f(x,v)\}})[i]$, where
  $q\colon M \times \R^{N+1} \to M\times\R$ is the projection and $i$ the index of
  the quadratic form which coincides with $f$ at infinity.  To see that this gives
  the correct sheaf we check using~\eqref{eq:borne-msupp} that
  $\dmsupp(F) \subset \Lambda$ (the map $q$ is not proper, but, since $f$ is a
  fibration at infinity, we can check directly that~\eqref{eq:borne-msupp} still
  holds here); we also see that $F_+ \simeq \ak_M$ and we conclude by the uniqueness
  property that this $F$ is the same as in the statement of the lemma.  Now
  $H^i(M\times \mo]-\infty,t\mc[; F) \simeq H^i(\{f<t\},\ak)$ and we recover the
  original definition of $c(\Lambda,\alpha)$ in~\cite{Vit92}.

  When $\Lambda$ is not given by a generating function, we can still use the
  construction of $F$ in~\cite{V19} to see that we obtain the definition of
  $c(\Lambda,\alpha)$ by Floer homology.
\end{proof}

\begin{remark}
  Let $f \colon M \to N$ be a proper map and $f' = f\times \id_\R$.
  By~\eqref{eq:borne-msupp} we see that $\RR f'_*F$
  satisfies~\eqref{eq:cond_faiscMR} if $F$ does. We also have
  $( \RR f'_*F)_+ \simeq \RR f_*F_+$ and we can identify
  $ H^\star(N;( \RR f'_*F)_+)$ with $H^\star(M;F_+)$.  Since
  $\rsect(M\times \mo]-\infty,t\mc[; F) \simeq \rsect(N\times \mo]-\infty,t\mc[; \RR
  f'_*F)$, we have
  \begin{equation*}
    c(F,\alpha) =  c(\RR f'_*F,\alpha) .
  \end{equation*}
  In particular, when $N$ is a point we have $c(F,\alpha) = c(\RR a_*F,\alpha)$,
  where $a \colon M\times\R \to \R$ is the projection.
\end{remark}

When $G = \RR a_*F$ is constructible (see Lemma~\ref{lem:im_directe_constr} for
a sufficient condition), we have a description of the numbers $c(G,\alpha)$ through
``barcodes''.  We will not need constructible sheaves in general and just recall what
it means over $\R$ (see~\cite{KS18} for more details).

An object $G \in \Der(\ak_{\R})$ is constructible if there exists finitely many
points $t_1$,\dots, $t_N \in \R$ such that the restriction of $G$ to each component of
$\R\setminus\{t_1,\dots,t_N\}$ is constant and moreover each stalk
$G_t \in \Der(\ak)$, $t\in \R$, is a bounded complex with finite dimensional
cohomology. It is not difficult to give an equivalence between the category of
constructible sheaves with respect to $\{t_1,\dots,t_N\}$ and a category of quiver
representations for a quiver of type $A_{2N+1}$.  Since we work with coefficient in a
field, we can apply Gabriel theorem and deduce that $G$ is a sum of constant sheaves
on intervals.  So we can write
$G \simeq \bigoplus_{I\in \mathcal{I}} \ak_I^{n_I}[d_I]$, where $\mathcal{I}$ is a
finite family of intervals (with ends belonging to
$\{-\infty,t_1,\dots,t_N,+\infty\}$), $n_I \in \N$ and $d_I\in \Z$.  If we assume
moreover, as will always be the case, that $\msupp(G) \subset \{\tau\geq 0\}$, we
know that these intervals are of the type $I= [a,b[$ with
$a\in \R\sqcup \{-\infty\}$, $b\in \R\sqcup \{+\infty\}$; in this case we call $G$ a
{\em barcode}.

Now we come back to $G = \RR a_*F$ where $F$ satisfies~\eqref{eq:cond_faiscMR}
and $G$ is assumed constructible.  By~\eqref{eq:cond_faiscMR} we have $G_t \simeq 0$
for $t\ll0$.  It follows from the above discussion that
\begin{equation}\label{eq:dec_barcode}
  G \simeq \bigoplus_{i \in \mathcal{I}_1} \ak_{[c_i, +\infty[}^{n_i}[d_i]
  \oplus \bigoplus_{j \in \mathcal{I}_2} \ak_{[a_j, b_j[}^{n_j}[d_j]
\end{equation}
with $a_i$, $b_i$, $c_i \in\R$, $n_i\in \N$ and $d_i\in \Z$.  Then
$H^\star(M; F_+) \simeq \bigoplus_{i \in \mathcal{I}_1} \ak^{n_i}[d_i]$ and the
numbers $c(\alpha, F)$ coincide with the $c_i$'s.  In particular
\begin{equation}
  \label{eq:c+c-cas_constr}
c_-(F) = \min\{c_i; \; i\in \mathcal{I}_1\}, \qquad
c_+(F) = \max\{c_i; \; i\in \mathcal{I}_1\} .  
\end{equation}

Let $\Lambda \subset T^*_{\tau>0}(M\times\R)$ be as above a $\R_{>0}$-conic closed
Lagrangian submanifold.  Let $F \in \Der(\ak_{M\times\R})$ with
$\dmsupp(F) = \Lambda$ and let $(z_0;\eta_0) \in \Lambda$.  We choose
$\varphi \colon M\times\R \to \R$ such that $\Gamma_{d\varphi} = \{(z;d\varphi(z))\}$
intersects $\Lambda$ transversely at $(z_0;\eta_0)$.  Then we know by Prop.~7.5.3
and Prop.~7.5.9 of~\cite{KS90} that, if $\Lambda$ is connected,
$(\rsect_{\{\varphi \geq \varphi(z_0)}(F))_{z_0} \in \Der(\ak)$ is independent of the
choice of $z_0$ and $\varphi$, up to a shift in degree.  With a suitable
normalization of the shift (see~\cite{KS90})
$(\rsect_{\{\varphi \geq \varphi(z_0)}(F))_{z_0}$ is called the {\em type} of $F$.
We say $F$ is of {\em finite type} if the total cohomology of its type is finite
dimensional. In~\cite{KS90} $F$ is called {\em simple} if the total cohomology of its
type is of dimension $1$.  The sheaf $F$ of Lemma~\ref{lem:cas_F_canonique} is
simple.

\begin{lemma}\label{lem:im_directe_constr}
  We assume that $\Lambda$ is a $\R_{>0}$-conic lift of an immersed compact exact
  Lagrangian submanifold $\ol \Lambda$ of $T^*M$ such that $\ol \Lambda$ meets $0_M$
  transversely.  Let $F \in \Der_{\Lambda,+}(\ak_{M\times\R})$ such that $F$ is of
  finite type in the above sense. Then $\RR a_*F$ is constructible.
\end{lemma}
\begin{proof}
  We set $G = \RR a_*F$.  Then $\dmsupp(G)$ is contained in the projection to
  $T^*\R$ of $L = \Lambda \cap (0_M \times T^*\R)$.  Since $\ol \Lambda$ meets $0_M$
  transversely, the intersection defining $L$ is also transverse and $L$ is a finite
  family of half-lines; this family is in bijection with $\ol \Lambda \cap 0_M$.
  Hence $\dmsupp(G) = \bigsqcup_{k=1}^N \{t_k\} \times \{\tau>0\}$, for some
  $N$. Moreover $G$ is of finite type by~\cite[Cor.~7.5.12]{KS90} at all points of
  $\dmsupp(G)$ (the type may vary since $\dmsupp(G)$ is not connected).

  In particular $G$ is constant on the intervals of
  $\R \setminus \{t_1,\ldots, t_N\}$ and it only remains to check that the stalks
  $G_t$ are finite dimensional. We have $G_t \simeq 0$ for $t\ll 0$.  Assuming the
  $t_k$'s are ordered, for $t\in [t_{k-1},t_k[$ and $s \in [t_k,t_{k+1}[$, we have a
  natural morphism $G_t \to G_s$ whose cone is the type of $G$ at $(t_k;1)$.  Hence
  $G_s$ is finite dimensional if $G_t$ is.  By induction we see that all stalks are
  finite dimensional.
\end{proof}

\begin{remark}\label{rem:position_generique}
  We give a series of remarks to check that we can restrict ourselves to the transverse case when
  we want to bound the spectral invariants.
  
\noindent  (1) For any given compact exact Lagrangian submanifold $\ol \Lambda$ of $T^*M$, we can find a
  compactly supported Hamiltonian isotopy $\phi^t$, $t\in [0,\varepsilon[$, such that $\phi^t(\ol
  \Lambda)$ meets $0_M$ transversely for $t\not=0$.

  \smallskip\noindent (2) We can lift the $\phi^t(\ol \Lambda)$'s into a family of conic Lagrangian
  $\Lambda_t$, $t\in [0,\varepsilon[$.  To a given $F \in \Der_{\Lambda,+}(\ak_{M\times\R})$ we can
  then associate a family $F_t \in \Der_{\Lambda_t,+}(\ak_{M\times\R})$ by~\cite{GKS12}.  These
  $F_t$ tend to $F$ with respect to a distance induced by the morphisms $\tau_c$
  of~\eqref{eq:deftau}, namely, there exist $a>0$ and morphisms $T_{-a t *}F_t \to[u_t] F \to[v_t]
  T_{a t *}F_t$ such that
  \begin{equation}
    \label{eq:convergenceFt}
v_t \circ u_t = \tau_{2a t}(T_{-a t *}F_t), \qquad    
T_{a t *}(u_t) \circ v_t = \tau_{2a t}(F).
\end{equation}
For the existence of $u_t$ and $v_t$ we choose $a$ big enough so that $t \mapsto \tilde\phi^{-t}
\circ T_{a t}$ and $t \mapsto T_{a t} \circ \tilde\phi^t$ are non negative isotopies, where
$\tilde\phi^t$ is a homogeneous lift of $\phi^t$ and we write abusively $T_c(x,t;\xi,\tau) =
(x,t+c;\xi,\tau)$, and we apply~\cite[Prop.~4.8]{GKS12}.

\smallskip\noindent (3) We have $(F_t)_+ \simeq F_+$ for all $t$. We can deduce
from~\eqref{eq:convergenceFt} that, for any
$\alpha \in H^\star(M;F_+) \simeq H^\star(M;(F_t)_+)$, the function
$t \mapsto c(\alpha, F_t)$ is continuous.
\end{remark}

\subsection{Propagation and torsion}

Now we assume $M$ is endowed with a metric. We denote by $||\cdot||_x$ (or
$||\cdot||$) the norms induced on the tangent and cotangent spaces $T_xM$, $T^*_xM$.
For $\alpha>0$ we define the conic subset of $T^*(M\times \R)$:
\begin{equation}
  \label{eq:defCalpha}
  C_\alpha = \{(x,t;\xi,\tau);\; \tau \geq \alpha\, ||\xi|| \} .
\end{equation}

\begin{lemma}\label{lem:SS_hom_etoile}
  Let $\alpha>0$ and $F, G \in \Der(\ak_{M\times\R})$ such that $\msupp(F)$ and
  $\msupp(G)$ are contained in $C_\alpha$.

  \suiv{(i)} Let $i \colon N \to M$ be the inclusion of a submanifold and
  $i_+ = i \times \id_\R$.  Then $\msupp(i_+^{-1}F) \subset C_\alpha$, where
  $C_\alpha$ is defined on $N\times\R$ using the induced metric. We also have
  $$
  i_+^{-1} \hom^*(F,G) \simeq \hom^*(i_+^{-1}F, i_+^{-1}G) .
  $$

\suiv{(ii)}  $\msupp( \hom^*(F,G)) \subset C_{\alpha/2}$.
\end{lemma}
\begin{proof}
  The first part of~(i) follows from~\eqref{eq:borne-msupp} and the second part
  from~\cite[Cor.~4.15]{GS14}.  The bound~(ii) is a particular case
  of~\cite[Prop.~4.13]{GS14}.
\end{proof}

\begin{proposition}\label{prop:annulation_section}
  Let $F \in \Der(\ak_{M\times\R})$ such that $\msupp(F)$ is contained in $C_\alpha$.
  Let $x_0\in M$ and $R>0$ be given with $R$ less than the injectivity radius of
  $M$. Let $B_R$ be the open ball centered at $x_0$ of radius $R$.  For
  $a<b' \leq b \leq 0$ we have the natural restriction morphisms
\begin{align*}
r_{x_0,]a,b[} &\colon H^0(B_R \times \mo]a,0[; F) \to
  H^0(\{x_0\}\times \mo]a,b[; F|_{\{x_0\}\times\R}), \\
  r_{]a,b'[} &\colon H^0(B_R \times \mo]a,0[; F) \to H^0(B_R \times \mo]a,b'[; F) .
\end{align*}
We assume that $b' < b-\alpha^{-1} R$. Then $\ker(r_{x_0,]a,b[}) \subset \ker(r_{]a,b'[})$.
\end{proposition}
\begin{proof} \newcommand{\Cbase}{{\mathcal C}} (i) We take
  $s\in \ker(r_{x_0,]a,b[})$ and prove that $s|_{B_R \times \mo]a,b'[}$ vanishes.  By
  Lemma~\ref{lem:propag1} the choice of $a$ is not relevant since, for example,
  $H^0(B_R \times \mo]a,b'[; F) \isoto H^0(B_R \times \mo]a',b'[; F)$ for any
  $a\leq a' <b'$ (at this point we could even take $a=-\infty$, but it is better to
  have a proper support to apply the Morse result below).  We choose $a'$ with
  $a<a'<b$.  Hence it is enough to see that $s|_{B_R \times \mo]a',b'[}$ vanishes.

  \medskip\noindent(ii) We define $\varphi(x) = d_M(x_0,x)$ on $M$ and we choose $0<\beta < \alpha$
  such that $b' < b - \beta^{-1}R$. For $\varepsilon >0$, we define the open ``cone''
  $\Cbase_\varepsilon = \{(x,t)\in B_R\times \R$; $t < b -\varepsilon - \beta^{-1} \varphi(x)\}$
  (see Fig.~\ref{fig:cone-etc}).  For $\varepsilon$ small enough we have $B_R \times \mo]a',b'\mc[
  \subset \Cbase_\varepsilon$ and it is enough to see that the restriction of $s$ to $(B_R \times
  \mo]a',0\mc[) \cap \Cbase_\varepsilon$ vanishes.

\begin{figure}[ht]
\begin{tikzpicture}[xscale=1.2]
\def\rayR{1.4} \def\rayd{.15} \def\rayb{-.9} \def\rayc{2}

\draw [dashed] (-\rayR,-4.5) rectangle (\rayR,0);
\draw [dashed] [fill=gray!10]
    (0,-.2) .. controls +(0,-.5) and +(0,.5) ..
    (-.5,-2) .. controls +(0,-.5) and +(-.2,.5) ..
    (0,-4.5) .. controls +(0.1,.5) and +(0,-.5) ..
    (.5,-3) .. controls +(0,1.5) and +(.2,-.5) .. 
    (0,-.2) ;
\draw [dashed] (-\rayd,-3.5) rectangle (\rayd,0);
\draw [dashed] (0,\rayb) -- +(\rayc,-\rayc);
\draw [dashed] (0,\rayb) -- +(-\rayc,-\rayc);
\begin{scope}
  \clip  (0,\rayb) -- +(\rayc,-\rayc) -- (\rayc,-5.4) -- (-\rayc,-5.4) -- (-\rayc,-\rayc+\rayb) -- cycle;
  \draw[rotate=45, step=5mm, color=gray!40] (-6,-6) grid (0,0);
\end{scope}
\draw [->] (0,-6) node[above right] {$\{x_0\}\times\R$} -- (0,1);
\draw (-2,0) -- (-\rayR,0 )  [->] (\rayR,0)  -- (2,0);
\draw [dotted] (0,-.2) -- +(3,0) node[right] {$b$};
\draw [dotted] (0,\rayb) -- +(3,0)  node[right] {$b-\varepsilon$};
\draw [dotted] (0,-3.5) -- +(3,0)  node[right] {$a'$};
\draw [dashed] (-\rayR,-2.5)  -- (\rayR,-2.5) ; 
\draw [dotted] (\rayR,-2.5) -- +(1.6,0)  node[right] {$b'$};
\draw [dotted] (\rayR,-4.5) -- +(1.6,0)  node[right] {$a$};
\node [above] at (\rayR,0) {$R$};
\node [above] at (\rayd,0) {$\delta$};

\node at (7,-2) {\begin{minipage}[t]{3.7cm} \setlength{\baselineskip}{8mm}
$U$ is shaded \\  $\Cbase_\varepsilon$ is striped\\
$\{x_0\}\times \mo]a,b\mc[ \subset U$ \\
$B_R \times \mo]-\infty,b'\mc[ \subset \Cbase_\varepsilon$ \\
$(B_\delta \times \mo]a',0\mc[) \cap \Cbase_\varepsilon \subset U$
\end{minipage}};
\end{tikzpicture}
\caption{}   \label{fig:cone-etc}
\end{figure}
  
Since $r_{x_0,]a,b[}(s) = 0$, there exists an open neighborhood $U$ of $\{x_0\}\times \mo]a,b[$ in
$B_R \times\R$ such that $s|_U$ vanishes (indeed $H^0(\{x_0\}\times \mo]a,b[; F|_{\{x_0\}\times\R})
\simeq \varinjlim_U H^0(U;F)$ where $U$ runs over such open neighborhoods --
see~\cite[Rem.~2.6.9]{KS90}).  Now we can find $\delta>0$ such that $U$ contains $(B_\delta \times
\mo]a',0\mc[) \cap \Cbase_\varepsilon$.  We visualize the restriction morphisms between the subsets
introduced so far in the diagram
  $$
  \begin{tikzcd}
    H^0(B_R \times \mo]a,0[; F)\ar[r]  \ar[rrr, bend left=20, "r_{x_0,]a,b[}"]
    \ar[d, "r_1"] \ar[dd, start anchor = south west, end anchor = north west, bend right=60,  "r_{]a',b'[}"']
    &[-10mm] H^0(U;F) \ar[dr] \ar[rr]
    &[-10mm] &[-50mm] H^0(\{x_0\}\times \mo]a,b[; F|_{\{x_0\}\times\R})\\
    H^0((B_R \times \mo]a',0\mc[) \cap \Cbase_\varepsilon; F) \ar[rr, "r_2"] \ar[d]
   && H^0((B_\delta \times \mo]a',0\mc[) \cap \Cbase_\varepsilon; F) \\
    H^0(B_R \times \mo]a',b'[; F) .    
  \end{tikzcd}
  $$
  We want to prove that $r_{]a',b'[}(s) = 0$. It is enough to see that $r_1(s) =0$.
  Since $s|_U = 0$, we have $r_2 \circ r_1(s) = 0$ and it is enough to see that $r_2$
  is an isomorphism.  Setting $V = (M \times \mo]a',0\mc[) \cap \Cbase_\varepsilon$
  and $\psi(x,t) = \varphi(x)$ we have
  $$
  H^0((B_{R'} \times \mo]a',0\mc[) \cap \Cbase_\varepsilon; F) \simeq
  H^0(\psi^{-1}(\mo]-\infty,R'[); \rsect_V(F)) 
  $$
  for any $R'$. By the microlocal Morse result in~\cite[Cor.~5.4.19]{KS90}, the
  restriction from $\psi^{-1}(\mo]-\infty,R[)$ to $\psi^{-1}(\mo]-\infty,\delta{}[)$
  is an isomorphism if
  \begin{equation}\label{eq:hypMorse}
    (x,t; d\varphi(x),0) \not\in \msupp(\rsect_V(F))
\qquad    \text{for all $x\in B_R \setminus B_\delta$.}
\end{equation}

\medskip\noindent(iii) We check~\eqref{eq:hypMorse} which concludes the proof of the
proposition.  We estimate $\msupp(\rsect_V(F))$ using~\eqref{eq:borne-msupp} (we
recall that $\rsect_V(F) \simeq \rhom(\ak_V,F)$).  When we restrict over
$(B_R \setminus B_\delta) \times \R$, the boundary of $V$ consists of two smooth
disjoint hypersurfaces, $M\times \{a'\}$ and $\partial \Cbase_\varepsilon$.  We only
check~\eqref{eq:hypMorse} at a point $(x,t) \in \partial \Cbase_\varepsilon$ (the
case $(x,t) \in V$ is obvious since $\rsect_V(F) \simeq F$ near such a point and the
case $(x,t) \in M\times \{a'\}$ is done in the same steps but easier).  Let $A$, $B$
be the intersection of $\msupp(\ak_V)$, $\msupp(F)$ with the fiber at $(x,t)$.  Then
$A = \R_{\geq 0} \cdot (d\varphi(x), \beta)$ and
$B \subset \{ \tau \geq \alpha\, ||\xi||_x \}$.  Since $|| d\varphi(x) ||_x =1$ and
$\beta < \alpha$, we have $A \cap B = \{0\}$ and, by~\eqref{eq:borne-msupp}, we
obtain the bound $A^a+B$ for $\msupp(\rsect_V(F))$ at $(x,t)$.  Hence,
if~\eqref{eq:hypMorse} does not hold, there exist $\lambda\geq 0$ and
$(\xi,\tau) \in B$ such that
$(d\varphi(x),0) = -\lambda (d\varphi(x), \beta) + (\xi,\tau)$, which gives
$\tau = \lambda\beta$, $\xi = (1+\lambda) d\varphi(x)$ and then
$\lambda\beta = \tau \geq \alpha||\xi||_x = \alpha(1+\lambda)$.  This contradicts
$\beta<\alpha$ and proves~\eqref{eq:hypMorse}.
\end{proof}

In the following corollary we make use of the remark: for any morphism
$i\colon N \to M$ (for example, $i$ is the inclusion of a submanifold or an open
subset) and $i_+ = i \times \id_\R$, the inverse image $i_+^{-1}$ sends
$\Der_{\tau\geq 0}(\ak_{M\times\R})$ to $\Der_{\tau\geq 0}(\ak_{N\times\R})$ and we
have $i_+^{-1} \circ T_{c*} = T_{c*} \circ i_+^{-1}$ and
$$
\tau_c(i_+^{-1}F) = i_+^{-1}(\tau_c(F)) .
$$

\begin{corollary}\label{cor:annulationtauc}
  Let $F \in \Der(\ak_{M\times\R})$.  We assume that $\msupp(F)$ is contained in
  $C_\alpha$ and that the morphism $\tau_c(F|_{\{x_0\} \times \R})$ vanishes for some
  $x_0\in M$ and some $c\geq0$.  Let $R>0$ be given, less than the injectivity radius
  of $M$. Let $B$ be the open ball centered at $x_0$ of radius $R$.  Then the
  morphism $\tau_{c'}(F|_{B \times \R})$ vanishes for $c'> c+2 \alpha^{-1} R$.
\end{corollary}
\begin{proof}
  We set $F_1 = \hom^*( F, F)[-1]$.  By Lemmas~\ref{lem:section_hom_etoile}
  and~\ref{lem:propag1}, the morphism $\id_F$ corresponds to a section
  $s \in H^1_{M\times \{0\}}(M\times \R; F_1) \simeq H^0(M \times \mo]a,0[; F_1)$,
  for any $a<0$, and we have to prove that the image of $s$ in
  $H^1_{B\times \{-c'\}}(B\times \R; F_1) \simeq H^0(B \times \mo]a,-c'[; F_1)$
  vanishes (assuming $a<-c'$).

  Lemma~\ref{lem:SS_hom_etoile} gives $\msupp(F_1) \subset C_{\alpha/2}$ and
  $F_1|_{\{x_0\}\times\R} \simeq \hom^*(F|_{\{x_0\}\times\R},
  F|_{\{x_0\}\times\R})[-1]$.  The hypotheses say that the image of $s$ in
  $H^0( \{x_0\} \times \mo]a,-c[; F_1|_{\{x_0\}\times\R})$ is zero.  Now the result
  follows from Proposition~\ref{prop:annulation_section}.
\end{proof}

\subsection{Barcodes computations}\label{sec:barcode}

We refer to the discussion before~\eqref{eq:dec_barcode} for the definition of barcode.

We want to compute $\hom^*(F,G)$ where $F,G$ are sums of constant sheaves on intervals $[a,b[$ with
$b$ being possibly infinite.  We give the following explicit formulas.

\begin{lemma}\label{lem:hom_etoile_bc} 
  Let $a,b,c,d$ be real numbers. Then
  \begin{align*}
  \hom^*(\ak_{[a,b[},\ak_{[c,d[})
  &\simeq \ak_{[c-b,\min(c-a,d-b)[}[1] \oplus \ak_{[\max(c-a,d-b),d-a[}, \\
\hom^*(\ak_{[a,\infty[},\ak_{[c,\infty[})&\simeq \ak_{]-\infty,c-a[}[1],\\
\hom^*(\ak_{[a,b[},\ak_{[c,\infty[})&\simeq \ak_{[c-b,c-a[}[1],\\
\hom^*(\ak_{[a,\infty[},\ak_{[c,d[})&\simeq \ak_{]c-a,d-a[} ,
  \end{align*}
  where $\ak_{[\alpha,\beta[} = 0$ if $\alpha \geq \beta$.
\end{lemma}
\begin{proof}
  We only prove the first two isomorphisms, the last two ones being similar to the
  second one.  Let us denote by $i\colon\R\to \R$ the involution $x\mapsto -x$ and
  $q_1,q_2\colon \R^2 \to \R$ the projections on the first and second factor
  respectively.  We use the formula
  $\hom^*(F,G) \simeq \RR s_*\rhom(q_2^{-1}i^{-1} F, q_1^!G)$ see~\cite{T08}
  or~\cite[Lem.~4.10]{GS14}.  We obtain
\begin{align*}
\hom^*(\ak_{[a,b[},\ak_{[c,d[})&= \RR s_*\rhom(q_2^{-1}i^{-1}\ak_{[a,b[},q_1^!\ak_{[c,d[})\\
                               &\simeq  \RR s_*\rhom(\ak_{\R \times ]-b,-a]},\ak_{[c,d[\times \R})[1]\\
                               &\simeq  \RR s_*\rhom(\ak_{\R \times ]-b,-a]},\rhom(\ak_{]c,d]\times \R},\ak_{\R^2}))[1]\\
                               &\simeq  \RR s_*\rhom(\ak_{\R \times ]-b,-a]}\otimes \ak_{]c,d]\times \R},\ak_{\R^2})[1]\\
                               &\simeq  \RR s_*\rhom(\ak_{ ]-b,-a]\times ]c,d]},\ak_{\R^2})[1]\\
                               &\simeq  \RR s_*(\ak_{ [-b,-a[\times [c,d[})[1]\ .
\end{align*}
We can easily compute local sections of this direct image: it amounts to computing
the cohomology of the constant sheaf on $E_I = ([-b,-a[\times [c,d[) \cap s^{-1}(I)$
where $I$ is an (open) interval of $\R$.  (To get the stalk at $t$, we choose
$I = \{t\}$ and $E_I$ is an interval which is closed, open or half-closed according
to $t$.)  This gives the first formula.

We perform a similar computation to prove the second formula:
\begin{align*}
  \hom^*(\ak_{[a,\infty[},\ak_{[c,\infty[})
  &\simeq  \RR s_*\rhom(\ak_{\R \times ]-\infty,-a]}\otimes \ak_{]c,\infty]\times\R},\ak_{\R^2})[1]\\
  &\simeq  \RR s_*\rhom(\ak_{ ]-\infty,-a]\times ]c,\infty]},\ak_{\R^2})[1]\\
  &\simeq  \RR s_*(\ak_{ [-\infty,-a[\times [c,\infty[})[1] 
\end{align*}
and we conclude by the same argument as for the first isomorphism.
\end{proof}

We now introduce a convenient bound for the spectral norm $c_+(F)-c_-(F)$ (see
Lemma~\ref{lem:c_et_v} below).

\begin{definition}\label{def:bounddepth}
  For $F \in \Der_{\tau\geq 0}(\ak_\R)$ we define its boundary depth
\begin{equation}\label{eq:betaettau}
  \beta(F)=\min\{c;\; \tau_c(F)=0 \} 
\end{equation} 
and, for $F,F' \in \Der_{\tau\geq 0}(\ak_\R)$ we define
  \begin{align*}
V(F,F') & = \hom^*(F,F') \otimes \ak_{[0,\infty[} , \\    
 v(F,F') &=\beta(V(F,F')) . 
  \end{align*}
  For a manifold $M$ and $F,F' \in \Der_{\tau\geq 0}(\ak_{M\times\R})$ we also set
  $v(F,F') = v(\RR a_*F, \RR a_*F')$, where $a\colon M\times\R \to \R$ is the
  projection.
\end{definition}

\begin{remark}
  A sheaf $F \in \Der_{\tau\geq 0}(\ak_\R)$ satisfying~\eqref{eq:cond_faiscMR} fits
  in a distinguished triangle
\begin{equation}
  \label{eq:triangle-condfaiscMR}
  F_0 \to F \to E_{[A,\infty[} \to[+1] ,
\end{equation}
where $\supp(F_0)$ is compact and $E$ is a complex of vector spaces.  If $F_1$ and
$F_2$ are either with compact support or of the form $E_{[A,\infty[}$, it is easy to
check that $V(F_1,F_2)$ is torsion (see
Definition~\eqref{def:tauc}). Using~\eqref{eq:triangle-condfaiscMR} we deduce that
$V(F,F')$ is torsion for $F,F'$ satisfying~\eqref{eq:cond_faiscMR}, hence that
$v(F,F')$ is finite.
\end{remark}

\begin{remark}
  If $F \in \Der_{\tau\geq0}(\ak_\R)$ is constructible, we have seen in the discussion
  before~\eqref{eq:dec_barcode} that $F \simeq \bigoplus_{I \in \mathcal{I}} \ak_I$, where
  $\mathcal{I}$ is a finite family of intervals $[a,b[$, $a$ and $b$ possibly infinite.  The
  boundary depth of $\mathcal{I}$ is usually defined as the longest finite interval in this
  decomposition.  When the barcode $F$ contains no infinite bar, it is torsion and the boundary
  depth of $\mathcal{I}$ coincides with $\beta(F)$.
\end{remark}

\begin{remark}\label{rem:imdir_somme_v}
  (1) It is clear from the definition that, for a morphism $f\colon M \to N$, we have
  $v(\RR(f\times\id_\R)_*(F), \RR(f\times\id_\R)_*(F')) = v(F,F')$.

  \sui(2) For $F_1, F_2 \in \Der_{\tau\geq 0}(\ak_\R)$ we have
  $V(F_1\oplus F_2, F') \simeq V(F_1, F') \oplus V(F_2, F')$. Hence
  $v(F_1,F') \leq v(F_1\oplus F_2, F')$. In the same way,
  $v(F',F_1) \leq v(F',F_1\oplus F_2)$.
\end{remark}

\begin{lemma}\label{lem:v_cone}
  Let $F_1,F_2,F,G \in \Der^{\perp,l}_{\tau\leq 0}(\ak_\R)$.  We assume that there
  exists a distinguished triangle $F_1 \to F_2 \to F \to[+1]$.  Then
  $v(F,G) \leq v(F_1,G) + v(F_2,G)$ and $v(G,F) \leq v(G,F_1) + v(G,F_2)$.
\end{lemma}
\begin{proof}
  We have distinguished triangles $V(G,F_1) \to V(G,F_2) \to V(G,F) \to[+1]$ and
  $V(F,G) \to V(F_2,G) \to V(F_1,G) \to[+1]$.  Now the result follows from
  Lemma~\ref{lem:taucducone}.
\end{proof}

In order to give a motivation to the definition of $v(F,F')$, it is interesting to compute the
special case $v(F,F)$. According to the bound of Lemma~\ref{lem:hom_etoile_bc}, in the transverse
case, $v(F,F)$ is an upper bound of the Viterbo spectral distance and of the boundary depth of the
couple $(F,\ak_{[0,\infty[})$. We check that this holds in the general case.  As in \S\ref{sec:Vsi}
we let $\ol \Lambda \subset T^*M$ be a compact exact Lagrangian submanifold and $\Lambda \subset
T^*_{\tau>0}(M\times\R)$ a $\R_{>0}$-conic lift of $\ol\Lambda$.

\begin{lemma}\label{lem:c_et_v}
  Let $F \in \Der_{\Lambda,+}(\ak_{M\times\R})$. We assume that $F$ is of finite
  type.  Then $c_+(F) - c_-(F) \leq v(F,F)$.  
\end{lemma}
\begin{proof}
  By Remark~\ref{rem:position_generique} we can assume that $\ol \Lambda$ meets $0_M$ transversely.
  By Lemma~\ref{lem:im_directe_constr} $G :=\RR a_* F$ is constructible, hence $G \simeq
  \bigoplus_{i \in \mathcal{I}_1} \ak_{[c_i, +\infty[}^{n_i}[d_i] \oplus \bigoplus_{j \in
    \mathcal{I}_2} \ak_{[a_j, b_j[}^{n_j}[d_j]$ as in~\eqref{eq:dec_barcode}.  As
  in~\eqref{eq:c+c-cas_constr} the family $\mathcal{I}_1$ contains $i_\pm$ such that $c_{i_\pm} =
  c_\pm(F)$ and $d_{i_-} = 0$, $d_{i_+} = -\dim M$.  It then follows from
  Lemma~\ref{lem:hom_etoile_bc} that $V(G,G)$ decomposes as a sum of sheaves of the type
  $\ak_{[a,b[}[c]$ which contains $\ak_{[0,c_{i_+} - c_{i_-}[}[-\dim M]$ as one summand.  The result
  follows.
\end{proof}

\section{Statements}

We give bounds for the spectral norm in the case where $M = \gp$ is a compact Lie
group and $M = \gp/\sgp$ where $\sgp$ is a Lie subgroup of $\gp$.  We endow $\gp$
with a bi-invariant metric and put the induced metric on $\gp/\sgp$.  In both cases
we denote by $B_1(M)$ the unit ball bundle in $T^*M$.  Let
$\ol\Lambda \subset B_1(M)$ be an immersed compact exact Lagrangian submanifold and
let $\Lambda \subset \dT^*(M\times\R)$ be a Legendrian lift of $\ol\Lambda$, seen as
a conic Lagrangian submanifold contained in $\{\tau>0\}$.

We set $n = \dim \gp$, $m = \dim \gp/\sgp$ and we let $l$ be the diameter of $\gp$
and $l_{max}$ the length of the maximal Reeb chord of $\Lambda$.  We recall the
notations $\Der_{\Lambda,+}(\ak_{M\times\R})$ of~\eqref{eq:defDerLambdaplus} and
$v(F,F')$ of Definition~\ref{def:bounddepth} which is a bound for the spectral norm.

\begin{theorem}\label{thm:groupe}
  For any $F,F' \in \Der_{\Lambda,+}(\ak_{\gp\times\R})$ of finite type we have
  $$
v(F,F') \leq (n+1)(2l + l_{max}) .
  $$
\end{theorem}

\begin{corollary}\label{cor:esphomogsimplconnexe}
  If $\sgp$ is not connected, we assume moreover that the characteristic of $\ak$
  does not divide $|\pi_0(\sgp)|$.  Then, for any
  $F,F' \in \Der_{\Lambda,+}(\ak_{\gp/\sgp\times\R})$ of finite type, we have
$$
v(F,F') \leq \frac14 (m+3)^2 (n+1)(2l + l_{max}) .
$$
\end{corollary}

We give the proof of Theorem~\ref{thm:groupe} in \S\ref{sec:Lie_group} and deduce
Corollary~\ref{cor:esphomogsimplconnexe} in \S\ref{sec:esp_homog}.

For a general $\Lambda$ it may happen that there is no sheaf in
$\Der_{\Lambda,+}(\ak_{\gp/\sgp\times\R})$.  In good cases we have a canonical $F$
and deduce results on the spectral norm of $\Lambda$.

\begin{corollary}\label{cor:defLegendrienne}
  We assume moreover that $\Lambda$ is a Legendrian deformation in $J^1(M\times \R)$
  of the Legendrian lift of an embedded compact exact Lagrangian submanifold of
  $T^*M$. Then
  $$
  c_+(\Lambda) - c_-(\Lambda) \leq
  \begin{cases}
(n+1)(2l + l_{max}) & \; \text{if $M =\gp$,} \\ 
\frac14 (m+3)^2 (n+1)(2l + l_{max}) & \; \text{if $M =\gp/\sgp$
  and $\operatorname{char}(\ak) \nmid |\pi_0(\sgp)|$.}     
  \end{cases}
  $$
\end{corollary}
\begin{proof}
  By Remark~\ref{lem:equiv-cat} there exists a unique
  $F \in \Der_{\Lambda,+}(\ak_{M\times\R})$ such that $F_+ \simeq \ak_M$.
  By~\cite{G19} this sheaf is simple, in particular of finite type.  Now the result
  follows from Lemmas~\ref{lem:cas_F_canonique} and~\ref{lem:c_et_v}.
\end{proof}

\section{Proof of Theorem~\ref{thm:groupe} -- compact Lie group case}
\label{sec:Lie_group}

We assume in this section that $M = \gp$ is a compact Lie group.  We let
$$
p_1,p_2, \mu \colon \gp \times \R \times \gp \times \R \to \gp \times \R,
\qquad a\colon \gp\times\R \to \R
$$
be the projections and the action defined by $p_1(g,t,g',t') = (g',t')$,
$p_2(g,t,g',t') = (g',t')$, $\mu(g,t,g',t') = (g g',t+t')$ and $a(g,t) =t$.  We will
use the following variation of the functor $\hom^*$ (see~\eqref{eq:defhometoile})
which takes into account the group structure.  For $F,F' \in \Der(\ak_{\gp\times\R})$
we define a sheaf on $\gp\times\R$:
\begin{equation}
  \label{eq:defhometoileG}
\hom^{*,\gp}(F,F') = \RR p_{1*}\rhom(p_2^{-1}F, \mu^!F') .
\end{equation}
This defines a functor $\hom^{*,\gp}$ which comes with a left adjoint $*_\gp$
defined by $F *_\gp F' = \RR \mu_!(p_1^{-1}F \otimes p_2^{-1}F')$:
\begin{equation}
    \label{eq:adjhometoileG}
\Hom(F,  \hom^{*,\gp}(F',F'')) \simeq \Hom(F *_\gp F' ,F'') .   
\end{equation}

We have seen in Section~\ref{sec:barcode} how to obtain some information about the spectral
invariants from $\hom^*( \RR a_*F, \RR a_*F')$.  In Proposition~\ref{projhomgroupe} we express this
latter sheaf as the direct image of $\hom^{*,\gp}(F,F')$ by $a$ and in Lemmas~\ref{lem:msuppF2},
\ref{lem:annulationF2fibree} we check that $\rsect_{\gp\times ]0,+\infty[} \hom^{*,\gp}(F,F')$
satisfies the hypotheses of Corollary~\ref{cor:annulationtauc}, which will be used to bound
$v(F,F')$.  The fact that $\Lambda$ is the Legendrian lift of an immersed exact Lagrangian is
actually only used in Lemma~\ref{lem:HomFTcF}; the others results follows from $\ol\Lambda \subset
B_1(\gp)$.

\begin{proposition}\label{projhomgroupe}
  We have $\RR a_* \hom^{*,\gp}(F,F') \simeq \hom^*(\RR a_{*}F, \RR a_{*}F')$.
\end{proposition}
\begin{proof}
Let $a:\gp \times \R \to \R$ the canonical projection. 
We introduce some maps and commutative diagrams in order to create Cartesian squares and apply base changes:
$$
\begin{tikzcd}
  (\gp \times \R) \times (\gp \times \R) \ar[d,"p_1"] \ar[r,"a\times\id"] 
& \R \times (\gp \times \R) \ar[d, "\ol{p}_1"] \ar[r,"\id\times a"] & \R\times\R \ar[d,"q_1"]\\
\gp \times \R \ar[r, "a"] & \R \ar[r,"\id"]& \R \rlap{,}
 \end{tikzcd}
 $$
 $$
\begin{tikzcd}
  (\gp \times \R) \times (\gp \times \R) \ar[d,"p_2"] \ar[r,"a\times\id"] 
& \R \times (\gp \times \R) \ar[d, "\ol{p}_2"] \ar[r,"\id\times a"] \arrow[dr, phantom, "\square"] & \R\times\R \ar[d,"q_2"] \\
\gp \times \R \ar[r, "\id"] & \gp \times \R \ar[r,"a"]& \R \rlap,
 \end{tikzcd}
 $$
 $$ 
 \begin{tikzcd}
  (\gp \times \R) \times (\gp \times \R) \ar[d,"\mu"] \ar[r,"a\times\id"] \arrow[dr, phantom, "\square"]
& \R \times (\gp \times \R) \ar[d, "\ol{\mu}"] \ar[r,"\id\times a"] & \R\times\R \ar[d,"s"]\\
\gp \times \R \ar[r, "a"] & \R \ar[r,"\id"]& \R \rlap,
 \end{tikzcd}
 $$
 with $\ol{\mu}$ the sum over the $\R$ factors.  Using these notations we obtain the
 sequence of isomorphisms
\begin{align*}
\RR a_* \RR p_{1*}& \rhom(p_2^{-1}F,\mu^! F') 
\\
&\simeq  \RR q_{1*} \RR( \id\times a )_ * \RR (a\times \id)_ * \rhom((a\times \id)^{-1}\ol{p}_2^{-1}F,\mu^! F') \\
&\simeq \RR q_{1*} \RR( \id\times a )_ * \rhom(\ol{p}_2^{-1}F,\RR (a\times \id)_ *\mu^! F')
\\
&\simeq \RR q_{1*} \RR( \id\times a )_ * \rhom(\ol{p}_2^{-1}F,\ol{\mu}^!\RR a_ * F')
\\
&\simeq \RR q_{1*} \RR( \id\times a )_ * \rhom(\ol{p}_2^{-1}F,(\id\times a)^!s^!\RR a_ * F')
\\
&\simeq \RR q_{1*} \rhom(\RR( \id\times a )_ !\ol{p}_2^{-1}F,s^!\RR a_ * F')
\\
&\simeq \RR q_{1*} \rhom( q_2^{-1}\RR a_ *F,s^!\RR a_ * F') \ .
\end{align*}
\end{proof}

Let $(g;\gamma)$ be the coordinates on $T^*\gp$. We recall the cone
$C_\alpha = \{(g,t;\gamma,\tau);\; \tau \geq \alpha\, ||\gamma|| \}$ defined
in~\eqref{eq:defCalpha}.

\begin{lemma}\label{lem:msuppF2}
  For any $F,F' \in \Der_{\Lambda,+}(\ak_{\gp\times\R})$ we have
  $\msupp(\hom^{*,\gp}(F,F')) \subset C_1$.
\end{lemma}
\begin{proof}
  This is proved in the same way as (ii) of Lemma~\ref{lem:SS_hom_etoile}.
  By~\cite[Prop.~5.4.5]{KS90} we have $\msupp(p_2^{-1}F) \subset A$,
  $\msupp(\mu^!F') \subset B$, where
  \begin{align*}
    A &= \{(g_1,t_1,g_2,t_2; 0,0, \gamma, \tau); \;
        \tau \geq ||\gamma|| \} , \\
    B &= \{(g_1,t_1,g_2,t_2; \gamma_1, \tau_1 , \gamma_2, \tau_2); \;
        \tau_2=\tau_1, \; \exists \gamma \in T^*_{g_1g_2}\gp,\; \tau_1 \geq ||\gamma||, \\
      &\hspace{4cm} \gamma_1 = {}^t(d\mu^r_{g_2})(\gamma), \; 
        \gamma_2 = {}^t(d\mu^l_{g_1})(\gamma) \},
  \end{align*}
  where $\mu^r_h, \mu^l_h\colon \gp \to \gp$ are the right and left actions,
  $g \mapsto gh$, $g\mapsto hg$.  Since $A\cap B$ is contained in the zero section,
  we deduce $\msupp( \rhom(p_2^{-1}F, \mu^!F')) \subset A^a+B$,
  using~\cite[Prop.~5.4.14]{KS90}.  Since the metric is bi-invariant, we have the
  rough bound $B \subset \{\tau_1 \geq ||\gamma_1||\}$ and hence
  $A^a+B \subset \{\tau_1 \geq ||\gamma_1||\}$.  By~\cite[Prop.~5.4.4]{KS90}
  $\msupp(\hom^{*,\gp}(F,F'))$ is contained in the projection of $A^a+B$ to the first
  factor $T^*(\gp\times\R)$, which concludes the proof.
\end{proof}

\begin{lemma}\label{lem:HomFTcF}
  For any $F,F' \in \Der_{\Lambda,+}(\ak_{\gp\times\R})$ we have
  $\RHom(F, T_{-c*}F') \simeq 0$ for all $c> l_{max}$.
\end{lemma}
\begin{proof}
  By Lemma~\ref{lem:section_hom_etoile} we have $\RHom(F, T_{-c*}F') \simeq \rsect_{\{c\}}\RR
  a_*\hom^*(F,F')$. We also have $\RR a_*\hom^*(F,F') \simeq \RR b_* \rhom(q_{2}^{-1} F, s^! G)$,
  where $b = a \circ q_1 \colon \gp\times \R^2 \to \R$, $(g,t_1,t_2) \mapsto t_1$ and $q_1$, $q_2$,
  $s$ are the notations of~\eqref{eq:defhometoile}.  We give a bound for the microsupport of $\RR
  a_*\hom^*(F,F')$.  By~\eqref{eq:borne-msupp} $\msupp(\rhom( q_{2}^{-1} F, s^! G)) \subset A^a+B$,
  where
  \begin{align*}
    A &= \{(g,t_1,t_2; \gamma, 0, \tau); \; (g,t_2;\gamma,\tau) \in \Lambda \}
        \cup 0_{\gp\times\R^2} , \\
    B &= \{(g,t_1,t_2; \gamma', \tau', \tau'); \; 
        (g,t_1+t_2;\gamma',\tau') \in \Lambda \}  \cup 0_{\gp\times\R^2} 
  \end{align*}
  and we are interested in the $(t_1;\tau_1) \in T^*\R$ such that
  $(g,t_1;t_2;0,\tau_1,0) \in A^a+B$ for some $g$, $t_2$.  Thus
  $0 = -\gamma + \gamma'$, $\tau_1 = \tau'$, $0 = -\tau+\tau'$ and, assuming
  $\tau\not=0$, we find $(g,t_2;\gamma,\tau) \in \Lambda$,
  $(g,t_1+t_2;\gamma,\tau) \in \Lambda$.  This implies that $\Lambda$ meets
  $T_{t_1}(\Lambda)$.  Hence $\dmsupp(\RR a_*\hom^*(F,F'))$ does not meet $T^*_c\R$
  when $c\not=0$ or $|c|$ is not a Reeb chord length.
  
  In particular $\RR a_*\hom^*(F,F')$ is constant on $]l_{max},\infty[$ and it is
  enough to check that $\RHom(F, T_{-c*}F') \simeq 0$ for $c\gg0$.

  For $c$ big enough, $T_{-c*}F'$ is locally constant on the support of $F$
  by~\eqref{eq:cond_faiscMR}. Hence we can assume $T_{-c*}F' \simeq p_\gp^{-1}(L')$
  for some $L' \in \Der(\ak_\gp)$, where $p_\gp$ is the projection to $\gp$.  Then
  $\RHom(F, T_{-c*}F') \simeq \RHom(\RR p_{\gp !}F, L')[-1]$.  Now we check that
  $\RR p_{\gp !}F$ vanishes.  By base change
  $(\RR p_{\gp !}F)_g \simeq \rsect_c(\R; F|_{\{g\} \times\R})$ for any $g\in \gp$.
  Since $F^g := F|_{\{g\} \times\R}$ is constant on $]A,\infty[$, $A\gg0$, we have
  $\rsect_c(\R; F^g_{[A,\infty[}) \simeq 0$.  Using the distinguished triangle
  $F^g_{]-\infty,A[} \to F^g \to F^g_{[A,\infty[} \to[+1]$, we obtain
  $(\RR p_{\gp !}F)_g \simeq \rsect(\R; F^g_{]-\infty,A[})$, which vanishes by
  Lemma~\ref{lem:annulationR}.
\end{proof}

\begin{lemma}\label{lem:annulationF2fibree}
  Let $e$ be the neutral element of $\gp$.  For any
  $F,F' \in \Der_{\Lambda,+}(\ak_{\gp\times\R})$ we have
  $\hom^{*,\gp}(F,F')|_{\{e \} \times [l_{max} ,\infty[} \simeq 0$.
\end{lemma}
\begin{proof}
  We set for short $F_1 = \hom^{*,\gp}(F,F')$.  Let $i_c \colon \{c\} \to \R$,
  $i_e \colon \R \simeq \{e\}\times\R \to \gp\times\R$ be the inclusions and
  $i_{(e,c)} = i_e \circ i_c$.  By Lemma~\ref{lem:msuppF2} $\dmsupp(F_1)$ does not
  meet $T_{\{e\}\times\R}(\gp\times\R)$. Hence
  $\msupp( i^{-1}_eF_1) \subset \{\tau\geq0\}$ and $i^{-1}_eF_1 \simeq i^!_eF_1 [n]$,
  where $n$ is the dimension of $\gp$, by~\cite[Prop.~5.4.13]{KS90}.  By
  Lemma~\ref{lem:cogermeszero} it is thus enough to prove that
  $i^!_c(i^{-1}_eF_1 ) \simeq i_{(e,c)}^! F_1 [n]$ vanishes for all $c> l_{max}$.

  By the adjunction~\eqref{eq:adjhometoileG} we have
  \begin{align*}
    i_{(e,c)}^! \hom^{*,\gp}(F,F')
    &\simeq \RHom(\ak_{(e,c)} ,  \hom^{*,\gp}(F,F')) \\
    &\simeq \RHom(\ak_{(e,c)} *_\gp F, F') \\
    &\simeq \RHom( T_{c*}F, F') \\
    &\simeq \RHom(F, T_{-c*}F') .
  \end{align*}  
Now the result follows from Lemma~\ref{lem:HomFTcF}.
\end{proof}

Until the end of the section we set for short
$$
F_2 = \hom^{*,\gp}(F,F')  \otimes \ak_{\gp \times [0,\infty[} .
$$
By Lemma~\ref{lem:annulationF2fibree} we thus have $\tau_c(F_2|_{\{e\}\times\R}) = 0$ for $c\geq
l_{max}$.

\begin{proposition}\label{prop:annultaugroupe0}
{\rm(i)}  For any $g\in\gp$ the morphism $\tau_c(F_2|_{\{g\} \times \R})$ vanishes for
  $c> 2 d(e,g) +l_{max}$, where $d(-,-)$ is the distance on $\gp$.

\noindent{\rm(ii)}  
Let $l$ be the diameter of $\gp$.  Let $\varepsilon>0$ be less than the injectivity
radius of $\gp$ and let $B \subset \gp$ be a ball of radius $<\varepsilon$. Then the
morphism $\tau_c(F_2|_{B \times \R})$ vanishes for
$c> 2(l+\varepsilon) + l_{max}$.
\end{proposition}
\begin{proof}
  (i) We choose a sequence of points $g_0,\dots,g_n$ on a geodesic in $\gp$ such that
  $g_0 = e$, $g_n=g$, $l_i := d(g_i, g_{i+1})$ is less than the injectivity radius
  and $\sum_{i=0}^{n-1} l_i = d(e,g)$.  We choose $\delta >0$ and set
  $c_i = l_{max} + 2 \sum_{j=0}^i (l_j+\delta)$.  Using
  Corollary~\ref{cor:annulationtauc} we see by induction on $i$ that
  $\tau_{c_i}(F_2|_{B(g_i,l_i+\delta)\times\R})$ vanishes and hence
  $\tau_{c_i}(F_2|_{\{g_{i+1}\}\times\R})$ also vanishes: the initial step is given
  by Lemma~\ref{lem:annulationF2fibree} and, by Lemma~\ref{lem:msuppF2}, the
  coefficient $\alpha$ in the corollary is $1$.  Since $\delta$ is arbitrary, the
  result follows.

  \smallskip\noindent(ii) Let $g$ be the center of $B$. By~(i)
  $\tau_c(F_2|_{\{g\} \times \R})$ vanishes for $c>2l + l_{max}$ and the result
  follows by Corollary~\ref{cor:annulationtauc} again.
\end{proof}

\begin{corollary}\label{cor:annulationtaucF2}
  The morphism $\tau_c(F_2)$ vanishes for $c> (n+1)(2l + l_{max})$, where
  $l$ is the diameter of $\gp$ and $n$ its dimension.
\end{corollary}
\begin{proof}
  Let $\varepsilon>0$ be given and $c = 2(l+\varepsilon) + l_{max}$.  We
  choose a triangulation of $\gp$ such that all simplices are contained in balls of
  radius less than $\varepsilon$.  We denote by $\Sigma_k$ the set of simplices of
  dimension $k$ and by $S_k$ the union of the simplices of dimension $\leq k$.  For a
  subset $S$ of $\gp$ we set for short $S^+ = S\times\R$.
  
  By Proposition~\ref{prop:annultaugroupe0}, for any simplex $\sigma$, the morphism
  $\tau_c(F_2|_{\sigma^+})$ vanishes.  We remark that
  $$
  \Hom(F_2|_{\sigma^+}, T_{c*}(F_2|_{\sigma^+})) \simeq \Hom( (F_2)_{\sigma^+},
  T_{c*}((F_2)_{\sigma^+})).
  $$
  Hence $\tau_c((F_2)_{\sigma^+})$ vanishes.  Let us prove that
  $\tau_{(k+1)c}((F_2)_{S_k^+})$ vanishes, by induction of $k$.  For $k=0$, this
  follows from Proposition~\ref{prop:annultaugroupe0}.  The induction step follows
  from Lemma~\ref{lem:taucducone} applied to the excision distinguished triangle
  $\bigoplus_{\sigma \in \Sigma_k} (F_2)_{\sigma^+} \to (F_2)_{S_k^+} \to
  (F_2)_{S_{k-1}^+} \to[+1]$.
  
For $k=n$ we obtain the vanishing of $\tau_{(n+1)c}(F_2)$. Since
$\varepsilon$ is as small as required, this gives the result.
\end{proof}

\begin{proof}[Proof of Theorem~\ref{thm:groupe}]
  By definition $v(F,F') = \min\{c;\; \tau_c(F_3)=0 \}$, where
  $F_3 = \hom^*(\RR a_*F, \RR a_*F') \otimes \ak_{[0,\infty[}$.  Using
  Proposition~\ref{projhomgroupe} and the projection formula we have
  $F_3 \simeq \RR a_* F_2$. By Corollary~\ref{cor:annulationtaucF2} we deduce
  $\tau_c(F_3) =0$ when $c> (n+1)(2l + l_{max})$, which proves the
  theorem.
\end{proof}

\section{Proof of Corollary~\ref{cor:esphomogsimplconnexe} -- homogeneous spaces}
\label{sec:esp_homog}

In this section we assume that $M = \gp/\sgp$ where $\gp$ and $\sgp$ are compact Lie groups.

  \sui(i) We recall that $\Der_{\Lambda,+}(\ak_{M\times\R}) \to \Dlc(\ak_M)$,
  $F \mapsto F_+$, is an equivalence (see Remark~\ref{lem:equiv-cat}).  More
  precisely there exists a unique $F_0 \in \Der_{\Lambda,+}(\ak_{M\times\R})$ such
  that $(F_0)_+ \simeq \ak_M$ and, for any $F \in \Der_{\Lambda,+}(\ak_{M\times\R})$
  we have $F \simeq F_0 \otimes p^{-1} F_+$, where $p\colon M\times\R \to M$ is the
  projection.

  \sui(ii) Let $q \colon \gp\times\R \to \gp/\sgp\times\R = M\times\R$ be the
  quotient map.  The projection formula gives
  $\RR q_!(q^{-1}F) \simeq \RR q_!(q^{-1}F \otimes \ak_{\gp\times\R}) \simeq F \otimes
  L$, where $L = \RR q_!(\ak_{\gp\times\R})$.  By Theorem~\ref{thm:groupe} and
  Remark~\ref{rem:imdir_somme_v} we have, for any
  $F, F' \in \Der_{\Lambda,+}(\ak_{M\times\R})$,
  \begin{equation}\label{eq:borne-groupe}
    v(F \otimes L, F'\otimes L) = v(q^{-1}F, q^{-1}F') \leq C,
    \qquad \text{ where $C:= (n+1)(2l + l_{max})$.}
\end{equation}

  \sui(iii) We first assume that $\sgp$ is connected.  We set
  $$
  \calF_{L,\Lambda} = \{F \otimes L; \; F \in \Der_{\Lambda,+}(\ak_{M\times\R})\},
  \qquad
  \calF_{L,\lc} = \{ G \otimes L; \; G \in \Dlc(\ak_{M\times\R})\} .
  $$  
  By~(i) we have $\calF_{L,\Lambda} = \calF_{L,\lc} \otimes F_0$, using the notation
  of Lemma~\ref{lem:coneiteres0}.  Since $\sgp$ is connected, we have
  $H^0L \simeq \ak_{M\times\R}$.  By Lemma~\ref{lem:coneiteres1} there exists
  $L' \in \calF_{L,\lc}^{\lceil (m-1)/2 \rceil}$ such that $\ak_{M\times\R}$ is a
  direct summand of $L'$ (locally constant sheaves on $M\times\R$ are pull-back of
  sheaves on $M$ and we can consider we work on $M$, which is of dimension $m$).  By
  Lemma~\ref{lem:coneiteres0} it follows that there exists
  $F' \in \calF_{L,\Lambda}^{\lceil (m-1)/2 \rceil}$ such that $F_0$ is a direct
  summand of $F'$.  For $F \in \Der_{\Lambda,+}(\ak_{M\times\R})$ we can write
  $F \simeq F_0 \otimes G$, for some $G\in \Dlc(\ak_{M\times\R})$.  Then $F$ is a
  direct summand of an object of
  $(\calF_{L,\Lambda} \otimes G)^{\lceil (m-1)/2 \rceil}$ by
  Lemma~\ref{lem:coneiteres0} again.  We remark that
  $\calF_{L,\Lambda} \otimes G \subset \calF_{L,\Lambda}$.

  \sui(iv) By~(ii) we have $v(F, F') \leq C$ for any $F , F'\in \calF_{L,\Lambda}$.
  Let us prove by induction on $i$ that, for any $j\leq i$ and
  $F_i \in \calF_{L,\Lambda}^{(i)}$, $F_j \in \calF_{L,\Lambda}^{(j)}$, we have
  $v(F_i,F_j) \leq C (i+1)(j+1)$, $v(F_j,F_i) \leq C (i+1)(j+1)$.

  For $i=0$ this is known. Let us assume it holds for $i$ and let us pick
  $F \in \calF_{L,\Lambda}^{(i+1)}$ given by a distinguished triangle
  $F_i \to F_0 \to F \to[+1]$ with $F_i \in \calF_{L,\Lambda}^{(i)}$,
  $F_0 \in \calF_{L,\Lambda}^{(0)}$.  Now we pick $F' \in \calF_{L,\Lambda}^{(j)}$
  with $j\leq i+1$.

  If $j\leq i$ the induction hypothesis and Lemma~\ref{lem:v_cone} give
  $v(F,F') \leq v(F_i,F') + v(F_0,F') \leq C (i+2)(j+1)$, as required.  Similarly for
  $v(F',F)$.  These inequalities holds for any $F \in \calF_{L,\Lambda}^{(i+1)}$ and
  $F' \in \calF_{L,\Lambda}^{(i)}$.  Hence, for $F' \in \calF_{L,\Lambda}^{(i+1)}$ we
  obtain $v(F_i,F') \leq C (i+1)(i+2)$ and $v(F_0,F') \leq C (i+2)$.  Now we can
  again apply Lemma~\ref{lem:v_cone} to obtain $v(F,F')$ and $v(F',F)$ in the case
  $j=i+1$.

  \smallskip By~(iii) we deduce
  $v(F,F') \leq C (\lceil (m-1)/2 \rceil+1)^2 \leq \frac14 C (m+3)^2$ for any
  $F, F' \in \Der_{\Lambda,+}(\ak_{M\times\R})$, which proves the corollary in the
  case $\sgp$ connected.

  \sui(v) If $\sgp$ is not connected, we let $\sgp^0$ be its neutral component and we
  set $M' = \gp/\sgp^0$.  Then $r\colon M' \to M$ is a finite cover with group
  $\pi_0(\sgp)$.  As in~(i) we have $r_!(r^{-1}F) \simeq F \otimes L_0$, where
  $L_0 = r_!(\ak_{M'})$. We remark that $r_! \simeq r_*$ is exact and
  $r^{-1} \simeq r^!$ because $r$ is a cover map with finite fibers.  We deduce the
  adjunction morphisms $\alpha\colon \ak_M \to L_0$ and $\beta \colon L_0 \to \ak_M$.
  The composition $\beta \circ \alpha$ is the multiplication by $|\pi_0(\sgp)|$.  By
  the hypothesis on the characteristic we deduce that $\ak_M$ is a direct summand of
  $L_0$. Now the result follows from~(iv) applied to $r^{-1}F$, $r^{-1}F'$ and
  Remark~\ref{rem:imdir_somme_v}.

\section{Appendix}

Let $\calT$ be a triangulated category.  For a family $\calF$ of objects of $\calT$
we set
$$
\calF^{(0)} = \{ F \in \calT;\; F \simeq \bigoplus_{i=1}^nF_i[d_i]
\text{ for some } F_i \in \calF, \, d_i\in \Z,\, i=1,\dots,n \}
$$
and we define $\calF^{(k)}$ inductively by
\begin{align*}
  \calF^{(k+1)} &= \{ F \in \calT;\; \text{there exists a distinguished triangle}\\
                & \hspace{3cm}  \text{$F_1 \to F_0 \to F \to[+1]$ with
                  $F_1 \in \calF^{(k)}$ and  $F_0 \in \calF^{(0)}$} \} .
\end{align*}
We remark that $(\calF^{(k)})^{(0)} = \calF^{(k)}$.

\begin{lemma}\label{lem:coneiteres0}
  Let $\calF$ be a family of objects of $\Der(\ak_M)$ and let $G \in \Der(\ak_M)$. We
  set $\calF \otimes G = \{ F\otimes G$; $F \in \calF\}$.  Then, for any
  $F' \in \calF^{(i)}$, we have $F'\otimes G \in (\calF\otimes G)^{(i)}$.
\end{lemma}
\begin{proof}
  We make an induction on $i$.  By definition any $F' \in \calF^{(i)}$ fits in a
  distinguished triangle $F_1 \to F_0 \to F' \to[+1]$ with $F_1 \in \calF^{(i-1)}$,
  $ F_0 \in \calF^{(0)}$.  Tensoring with $G$ we obtain
  $F_1\otimes G \to F_0\otimes G \to F'\otimes G \to[+1]$ where
  $F_1\otimes G \in (\calF\otimes G)^{(i-1)}$ by the induction hypothesis and
  $F_0\otimes G \in (\calF\otimes G)^{(0)}$. Hence
  $F'\otimes G \in (\calF\otimes G)^{(i)}$.
\end{proof}

We let $\Dlc(\ak_M)$ be the subcategory of $\Der(\ak_M)$ formed by the $G$ such that
$H^iG$ is a locally constant sheaf, for any $i\in\Z$.

\begin{lemma}\label{lem:coneiteres1}
  Let $n$ be the dimension of $M$.  Let $L \in \Dlc(\ak_M)$ be such that
  $H^iL \simeq 0$ for $i<0$ and $H^0L \simeq \ak_M$.  We set $\calF = \{F$;
  $F \simeq L \otimes G$ for some $G \in \Dlc(\ak_M)\}$.  Then there exists
  $F \in \calF^{\lceil (n-1)/2 \rceil}$ such that $\ak_M$ is a direct summand of $F$.
\end{lemma}
\begin{proof}
  (i) For an integer $i$ we let $\Der^{\geq i}(\ak_M)$ be the subcategory formed by
  the $G$ such that $H^jG \simeq 0$ for $j<i$.  We say that
  $G \in \Der^{\geq 0}(\ak_M)$ is $l$-lacunary if $H^jG \simeq 0$ for $j=1,\dots,l$
  (any object of $\Der^{\geq 0}(\ak_M)$ is $0$-lacunary).  We recall the truncation
  distinguished triangle, for any $G \in \Der(\ak_M)$ and $i\in \Z$ (see for
  example~\cite[(1.7.2)]{KS90})
$$
\tau_{\leq i}G \to G \to \tau_{\geq i+1}G \to[+1] .
$$

\sui(ii) We first prove: let $F \in \Der^{\geq 0}(\ak_M)$ be $(n-1)$-lacunary and
such that $H^0F \simeq \ak_M$. Then $\ak_M$ is a direct summand of $F$.  Indeed we
have a distinguished triangle $\ak_M \to F \to F' \to[u] \ak_M[1]$ where
$F' = \tau_{\geq 1}F$. Since $F$ is $(n-1)$-lacunary, we have
$F' \in \Der^{\geq n}(\ak_M)$. We can see $u$ as a morphism from $F'[-1]$ to $\ak_M$.
Now $\ak_M$ has a flabby resolution of length $n$ and, since $\ak$ is a field, flabby
sheaves are injective (see~\cite{KS90} Ex.~II.9 and~II.10).  Hence we can compute $u$
by replacing $\ak_M$ with a complex in degrees $0,\dots,n$.  Since
$F'[-1] \in \Der^{\geq n+1}(\ak_M)$ we find $u=0$.  Hence $F \simeq \ak_M \oplus F'$,
as required.
  
  \sui(iii) We prove by induction on $i$ that there exists $L_i\in \calF^{(i)}$ such
  that $L_i \in \Der^{\geq 0}(\ak_M)$, $H^0L_i \simeq \ak_M$ and $L_i$ is
  $2i$-lacunary.  By~(ii) applied to $F=L_{\lceil (n-1)/2 \rceil}$ this proves the
  lemma.

  For $i=0$ we take $L_0 = L$.

  We assume that step $i$ is proved and we pick $L_i$ as above. We have the
  distinguished triangles
  \begin{align}
    \label{eq:coneiteres1}
    & \ak_M \to L \to L' \to[+1] , \\
    \label{eq:coneiteres2}
    & \ak_M \to L_i \to[\alpha] L'_i \to[+1] ,
  \end{align}
  where $L' = \tau_{\geq 1} L$ and $L'_i = \tau_{\geq 1} L_i$.
  Tensoring~\eqref{eq:coneiteres1} with $L'_i$ we obtain
  \begin{equation}
    \label{eq:coneiteres3}
   L'_i  \to[\beta] L\otimes L'_i \to L'\otimes L'_i \to[+1] 
  \end{equation}
  and we define $G$ by the distinguished triangle
  \begin{equation}
    \label{eq:coneiteres4}
   L_i  \to[\beta \circ \alpha] L\otimes L'_i \to G \to[+1] .
 \end{equation}
 We remark that $L'_i \in \Dlc(\ak_M)$, hence $L\otimes L'_i \in \calF$ and
 $G \in \calF^{(i+1)}$.  The octahedron axiom applied
 to~\eqref{eq:coneiteres2}-\eqref{eq:coneiteres4} gives the distinguished triangle
  \begin{equation}
    \label{eq:coneiteres5}
   \ak_M[1]  \to G \to  L'\otimes L'_i \to[+1] .
 \end{equation}
 We have $L' \in \Der^{\geq 1}(\ak_M)$ by definition and
 $L'_i \in \Der^{\geq 2i+1}(\ak_M)$ because $L_i$ is $2i$-lacunary.  Hence
 $L'\otimes L'_i \in \Der^{\geq 2i+2}(\ak_M)$ and
 $L'\otimes L'_i [-1] \in \Der^{\geq 2i+3}(\ak_M)$.  Now it follows
 from~\eqref{eq:coneiteres5} that $L_{i+1} = G[-1]$ satisfies the required
 properties.
\end{proof}

\providecommand{\bysame}{\leavevmode\hbox to3em{\hrulefill}\thinspace}


\begin{thebibliography}{00}


\bibitem{AI17} T.~Asano and Y.~Ike,
{\em Persistence-like distance on Tamarkin's category and symplectic displacement energy},
J. Symplectic Geom. {\bf 18}, No. 3, 613--649 (2020). 


\bibitem{BR} P.~Biran and O.~Cornea,
  {\em Bounds on the Lagrangian spectral metric in cotangent bundles},
 Commentarii Mathematici Helvetici,  Volume {\bf 96}, Issue 4,  631--691 (2021)

\bibitem{D} G. Dimitroglou Rizell,
{\em  Families of Legendrians and Lagrangians with unbounded spectral norm},
\texttt{arXiv:2012.15559}


\bibitem{G19} S.~Guillermou,
{\em  Sheaves and symplectic geometry of cotangent bundles},
\texttt{arXiv:1905.07341}

  
\bibitem{GKS12} S.~Guillermou, M.~Kashiwara and P.~Schapira,
{\em Sheaf quantization of Hamiltonian isotopies and applications to
non displaceability problems,}
Duke Math. J. {\bf 161} no. 2, 201--245  (2012).

\bibitem{GS14} S.~Guillermou and P.~Schapira,
{\em Microlocal theory of sheaves and Tamarkin's non displaceability theorem}
in Homological Mirror Symmetry and Tropical Geometry, 
edited by R.~Castano-Bernard et al.,
Lect. Notes of the UMI {\bf 15}, 43--85 (2014)

\bibitem{KS85} M.~Kashiwara and P.~Schapira,
{\em Microlocal study of sheaves,}
Ast{\'e}risque {\bf 128} Soc.\ Math.\ France (1985).

\bibitem{KS90} M.~Kashiwara and P.~Schapira,
{\em Sheaves on Manifolds,}
\/Grundlehren der Math. Wiss. {\bf 292} Springer-Verlag (1990).

\bibitem{KS18} M.~Kashiwara and P.~Schapira,
  {\em  Persistent homology and microlocal sheaf theory},
  Journal of Applied and Computational Topology, Vol. {\bf 2}, 83--113 (2018)

\bibitem{MVZ} A. Monzner, N. Vichery, F. Zapolsky,
  {\em Partial quasi-morphisms and quasi-states on cotangent bundles, and symplectic homogenization},
  Journal of modern dynamics, Volume {\bf 6}, No.2, 205--249, (2012)

\bibitem{N09} D.~Nadler,
{\em Microlocal branes are constructible sheaves,}
Selecta Math. (N.S.) {\bf 15}, 563--619 (2009).

\bibitem{N15} D.~Nadler,
{\em Non-characteristic expansions of Legendrian singularities},
\texttt{arXiv:1507.01513}.

\bibitem{NZ09} D.~Nadler and E.~Zaslow, 
{\em Constructible sheaves and the Fukaya category,}
J. Amer. Math. Soc. {\bf 22} 233--286 (2009).

\bibitem{S18} E.~Shelukhin, {\em Viterbo conjecture for Zoll symmetric spaces}, arXiv:1811.05552.

\bibitem{S19} E.~Shelukhin, {\em Symplectic cohomology and a conjecture of Viterbo}, arXiv:1904.06798.

\bibitem{T08} D.~Tamarkin, 
{\em Microlocal conditions for non-displaceability,}
in  Algebraic and analytic microlocal analysis,
edited by M.~Hitrik et al.,
Springer Proceedings in Mathematics and Statistics {\bf 269},
99--223 (2018)
\/\texttt{arXiv:0809.1584}.


\bibitem{Vit92} C.~Viterbo,
{\em Symplectic topology as the geometry of generating functions},
Math. Ann. {\bf 292}, No. 4, 685-710 (1992). 


\bibitem{Vit08} C.~Viterbo,
{\em  Symplectic Homogenization},
\texttt{arXiv:0801.0206 }. 
 
\bibitem{V11} C.~Viterbo,
{\em Notes on symplectic geometry and the proof of Arnold conjecture using sheaves,}
Lectures at Princeton Fall 2010 and New York Spring 2011,
\texttt{www.dma.ens.fr/\~{}viterbo}.

\bibitem{V19} C.~Viterbo,
{\em Sheaf Quantization of Lagrangians and Floer cohomology},
\texttt{arXiv:1901.09440}.

\bibitem{V22} C.~Viterbo,
{\em Inverse reduction inequalities for spectral numbers and applications},\\
\texttt{arXiv:2203.13172}.

\end{thebibliography}
\end{document}